\documentclass[11pt,leqno]{article}
\usepackage{mathrsfs,amssymb,amsmath,amsthm, color}
\usepackage[dvips]{graphicx}
\usepackage{mathtools}

\makeatletter
	
	\@addtoreset{equation}{section}
\makeatother

\begin{document}

\renewcommand{\theenumi}{\rm (\roman{enumi})}
\renewcommand{\labelenumi}{\rm \theenumi}

\newtheorem{thm}{Theorem}[section]
\newtheorem{defi}[thm]{Definition}
\newtheorem{lem}[thm]{Lemma}
\newtheorem{prop}[thm]{Proposition}
\newtheorem{cor}[thm]{Corollary}
\newtheorem{exam}[thm]{Example}
\newtheorem{conj}[thm]{Conjecture}
\newtheorem{rem}[thm]{Remark}
\newtheorem{ass}[thm]{Assumption}
\allowdisplaybreaks

\title{Berry-Esseen bounds for large-time asymptotics of one-dimensional diffusion processes via Malliavin-Stein method}

\author{Seiichiro Kusuoka and Yuichi Shiozawa}
\maketitle

\begin{abstract}
We consider solutions of stochastic differential equations 
which diverge to infinity as the time parameter goes to infinity.
If the coefficients converge as the spacial variable goes to infinity, 
then the solutions will get close to some Gaussian processes with positive drifts as the time parameter goes to infinity.
In this paper, we prove Berry-Esseen type bounds for the solutions in this setting.
In particular, we obtain bounds of the total variation distance 
between the law of the centered and scaled solutions of the stochastic differential equations 
and the standard normal distribution with an optimal rate of convergence in the time parameter.
In the proof we apply the Malliavin-Stein method to estimate the total variation distance.
\end{abstract}

\section{Introduction}

Let $l\in {\mathbb R}\cup \{ -\infty\}$ and consider the one-dimensional stochastic differential equation (SDE):
\begin{equation}\label{eq:SDE0}\left\{\begin{array}{l}
dX_t^x = \sigma (t,X_t^x) dB_t + b(t,X_t^x)dt, \\
X_0^x =x \in (l,\infty )
\end{array}\right.\end{equation}
where $\sigma , b \in C^{0,1}([0, \infty ) \times (l,\infty ))$.
Here, $C^{0,1}([0, \infty ) \times (l,\infty ))$ means the class of continuous functions on $[0, \infty ) \times (l,\infty )$ 
which are differentiable with respect to the second parameter (spatial parameter) and the derivative is continuous on $[0, \infty ) \times (l,\infty )$.
We assume that \eqref{eq:SDE0} has a pathwise-unique solution without explosions. 
When the limits $\sigma _\infty (t):= \lim _{y\rightarrow \infty} \sigma (t,y)$ and $b_\infty (t):= \lim _{y\rightarrow \infty} b(t,y)$ exist 
and $\lim _{t\rightarrow \infty} X_t^x = \infty$ almost surely, 
we expect that the asymptotic behaviour of $X_t^x$ for large $t$ is close to $\int _0^t \sigma _\infty (s) dB_s + \int _0^t b_\infty (s)ds$.
Such asymptotic behaviours are studied in the case that $X_t^x$ is the radial part of the Brownian motion on $d$-dimensional hyperbolic spaces. More precisely, some limit theorems like the law of large numbers, the central limit theorems and the Berry-Esseen type bound have been obtained 
(see \cite{Shi23} and the references therein).

In the present paper, we give suitable assumptions on the coefficients $\sigma$ and $b$, which guarantee the existence of the limits $\sigma _\infty (t)$ and $b_\infty (t)$ which are uniformly positive and bounded in $t$, and guarantee the divergence $\lim _{t\rightarrow \infty} X_t^x = \infty$ almost surely.
Under the assumptions we obtain the Berry-Esseen type bound:
\begin{equation}\label{eq:intro1}
d_{\rm  TV} \left( {\rm Law} \left(\frac{X_t^x -x - t \overline{b _\infty}(t)}{\overline{\sigma _\infty}(t) \sqrt{t}} \right) , {\rm Law}(Z)\right) \leq Ct^{-\frac{1}{2}}, \quad t\geq 1,
\end{equation}
where $d_{\rm  TV}$ is the total variation distance, $Z$ is a random variable with the standard normal law,
\[
\overline{\sigma _\infty}(t) := \sqrt{\frac{1}{t} \int _0^t \sigma _\infty (s)^2 ds} \quad \mbox{and} \quad \overline{b _\infty}(t) := \frac{1}{t} \int _0^t b_\infty (s) ds .
\]
We remark that \eqref{eq:intro1} immediately implies a version of the law of large numbers:
\[
\lim _{t\rightarrow \infty} \left( \frac{X_t^x}{t} - \overline{b _\infty}(t)\right) =0 \quad \mbox{almost surely},
\]
and a version of the central limit theorem:
\[
\mbox{the law of}\quad \frac{X_t^x -x - t \overline{b _\infty}(t)}{\overline{\sigma _\infty}(t) \sqrt{t}} \quad \mbox{converges to the standard normal distribution}.
\]
We also remark that our argument in the present paper is applicable to 
the radial part of the Brownian motion on hyperbolic spaces. 
In the case, the order $-\frac{1}{2}$ of $t$ in the upper bound of \eqref{eq:intro1} is optimal 
(see \cite{Shi23} and Section \ref{sec:relax} for details).

It is easy to obtain the law of large numbers and the central limit theorem above (see Theorems \ref{thm:main1} and \ref{thm:CLT}, respectively).
However, it is difficult in general  to obtain the Berry-Esseen type bound.
In particular, we need precise estimates for the optimal rate $t^{-\frac{1}{2}}$.
Indeed, we could not directly apply the ergodic theory to the process $\{X_t^x\}$, 
because its asymptotic behaviour is of a Gaussian process with a positive drift, and so the process $\{X_t^x\}$ is not stationary.
Moreover, our estimate is of the total variation distance (see \eqref{eq:intro1}), 
which is not easily treated by previous methods.

To obtain the bound of the total variation distance, we apply the Malliavin-Stein method and the comparison theorem of solutions of SDEs.
In the Malliavin-Stein method we have a complicated term (called Stein factor) in the upper bound, 
which includes the Malliavin derivative and the inverse of the Ornstein-Uhlenbeck operator (See Section \ref{sec:MalliavinStein} for details).
Because of this, it is difficult in general to simplify the bound.
However, in the present paper we are able to calculate the bound, 
because the asymptotic behaviour of $X_t^x$ in large $t$ is a Gaussian process with a nonrandom drift.
The error terms are estimated well via the comparison of the solution with a martingale with a constant positive drift.
Hence, the result in the present paper seems interesting, not only as a result of the asymptotic behaviour of the solutions of one-dimensional SDEs, 
but also as an application of the Malliavin-Stein method.

We first consider the case that $\sigma, b \in C_b^{0,1}([0,\infty )\times {\mathbb R})$ 
and some other suitable assumptions are fulfilled (See Section \ref{sec:BE}).
As mentioned above, we obtain \eqref{eq:intro1} by applying the Malliavin-Stein method, which is Theorem \ref{thm:BE1}, the first main theorem in the present paper. 
Hence, we need precise estimates of the terms, including Malliavin derivatives of the solution, appearing in the bound coming from the Malliavin-Stein method.
The main part of the proof of Theorem \ref{thm:BE1} is the proofs of the estimates, and the comparison theorem is applied several times.

We next consider the case that for some $l\in {\mathbb R}$, 
$\sigma \in C_b^{0,1}([0, \infty ) \times (l,\infty ))$, $b \in C^{0,1}([0, \infty ) \times (l,\infty ))$ 
and $\lim_{y\downarrow l}b(t,y)=\infty$ for any $t\ge 0$.
We impose some explicit conditions on the asymptotic bahaviours of $b(t,y)$ and $\partial b/\partial y(t,y)$ around $y=l$. 
The argument in this section works in the case that $X_t^x$ is the radial part of the Brownian motion on hyperbolic spaces (see Section \ref{sec:relax}).
In this case, we need estimates about the behaviour of the solution around $l$.
By using these and similar estimates to those in the case of bounded $b$, we obtain \eqref{eq:intro1}.
See Section \ref{sec:BEunbdd} for details.

Finally, we consider the case that some assumptions on $\sigma$ and $b$ are replaced by those on the assymptotic behaviours of $\sigma$ and $b$.
We apply the (time-inhomogeneous) strong Markov property of the solution  $\{X_t^x\}$ and the result in the case that $\sigma, b \in C_b^{0,1}([0,\infty )\times {\mathbb R})$.
In this case, we obtain a similar estimate to \eqref{eq:intro1} but the convergence rate $t^{-\frac{1}{2}}$ in the upper bound is replaced by $t^{-\frac{1}{2}}\log t$.
See Section \ref{sec:BEasypm} for details.

The organization of the present paper is as follows.
In Section \ref{sec:preliminary} we prepare estimates which will appear in the proof of main theorems.
More precisely, in Section \ref{sec:MalliavinStein} we recall the estimate of the total variation distance between a given law and the standard normal distribution in the Malliavin-Stein method.
In Section \ref{sec:Brownian} we give a known result and its implication on the exponential integrability of the functional of the Brownian motion with drift, 
and in Section \ref{sec:marti+drift} we extend the propositions in Section \ref{sec:Brownian} to the case of martingales with drifts.
In Sections \ref{sec:BE}, \ref{sec:BEunbdd} and \ref{sec:BEasypm} 
we consider the case that $\sigma ,b \in C_b^{0,1}([0,\infty )\times {\mathbb R})$, 
the case that $\sigma \in C_b^{0,1}([0, \infty ) \times (l,\infty ))$ and $b \in C^{0,1}([0, \infty ) \times (l,\infty ))$ diverges at $l\in {\mathbb R}$, 
and the case that $\sigma$ and $b$ have assumptions on their asymptotic behaviours, respectively. 
We obtain the main theorems in each section.
In Appendix we give an estimate of the total variation distance between a normal distribution and the scaled and translated one of the normal distribution. 

\subsection*{Notations}

Throughout this paper, all random variables and stochastic processes are defined on the probability space $(\Omega , {\cal F}, P)$, and we denote the expectation by $E[\cdot ]$.
For nonnegative real numbers $a$ and $b$, let $a\vee b := \max \{ a,b \}$ and $a\wedge b := \min \{ a,b\}$.

\section{Preliminary}\label{sec:preliminary}

\subsection{The Malliavin-Stein method}\label{sec:MalliavinStein}

Here, we recall the estimate of the distance between a given distribution and the standard normal distribution 
via the Malliavin-Stein method (see \cite{KT12, KT18, NP12} for details).

Let $A := \frac{d^2}{dx^2}-x\frac{d}{dx}$, $p(x) := \frac{1}{\sqrt{2\pi}} \exp \left( -\frac{x^2}{2}\right)$, 
and let $Z$ be a random variable with the standard normal distribution.
For $f\in C_b({\mathbb R})$, define
\begin{align*}
m_f &:= E[f(Z)] = \int _{\mathbb R} f(y) p(y) dy, \\
\widetilde{g_f}(x) &:= \frac{1}{p(x)} \int _{-\infty} ^x (f(y)-m_f) p(y) dy ,\\
g_f(x) &:= \int _0^x \widetilde{g_f}(y) dy.
\end{align*}
Note that in view of \cite[Proposition 2]{KT12} (or \cite[Theorem 3.3.1 and Exercise 3.3.4]{NP12}), $\widetilde{g_f} \in W^{1,\infty}({\mathbb R})$ and 
\begin{equation}\label{eq:estgf}
\| \widetilde{g_f} \| _{W^{1,\infty }} \leq C\| f\| _{L^\infty}.
\end{equation}
Then, we have
\[
A g_f (x) = \frac{d^2 g_f}{dx^2}(x) - x \frac{dg_f}{dx}(x) = \frac{d \widetilde{g_f}}{dx}(x)-x\widetilde{g_f}(x) = f(x) -m_f .
\]
Hence, for any random variable $X$, it holds that
\begin{equation}\label{eq:Stein}
E[f(X)] - E[f(Z)] = E\left[ \frac{d\widetilde{g_f}}{dx} (X) - X \widetilde{g_f}(X) \right] .
\end{equation}

Next we apply the Malliavin calculus.
Let $(B,H,\mu )$ be the abstract Wiener space, let $D$ be the Malliavin derivative, and for the unity of notations we let $\Omega =B$ and $P = \mu$.
Let $L:= -D^*D$ be the Ornstein-Uhlenbeck operator, and let ${\cal D}^{s,2}(H^{\otimes n})$ be the Sobolev space with the associated norm 
\[
\| X\| _{{\cal D}^{s,2}(H^{\otimes n})} := E\left[ \left\| (1-L)^{\frac{s}{2}} X \right\| _{H^{\otimes n}}^2 \right] ^{1/2}.
\]
Then, if $X\in {\cal D}^{1,2}$, we have
\begin{align*}
E\left[ X \widetilde{g_f}(X) \right] &= E\left[ D^* D (-L)^{-1} (X-E[X]) \widetilde{g_f}(X) \right] + E[X] E[\widetilde{g_f}(X) ]\\
&= E\left[ \frac{d\widetilde{g_f}}{dx} (X) \left\langle D (-L)^{-1} (X-E[X]),  DX \right\rangle _H \right] + E[X] E[\widetilde{g_f}(X) ].
\end{align*}
This equality and \eqref{eq:Stein} yield
\[
E[f(X)] - E[f(Z)] = E\left[ \frac{d\widetilde{g_f}}{dx} (X) \left( 1- \left\langle D (-L)^{-1} (X-E[X]),  DX \right\rangle _H \right) \right] - E[X] E[\widetilde{g_f}(X) ].
\]
Thus, together with \eqref{eq:estgf} we obtain
\begin{equation}\label{eq:MS}\begin{array}{l}
\displaystyle \left| E[f(X)] - E[f(Z)] \right| \\
\displaystyle \leq C \| f\| _{L^\infty} \left( E\left[ \left|  1- \left\langle D (-L)^{-1} (X-E[X]),  DX \right\rangle _H  \right| \right] + |E[X]| \right).
\end{array}\end{equation}
This inequality is obtained in \cite[Theorem 1]{KT12}.
As in \cite[Section 3]{KT12}, \eqref{eq:MS} enables us to estimate some distances between the laws of $X$ and $Z$.

Recall that the total variation distance between the laws of $X$ and $Z$ is given by
\[
d_{\rm  TV}({\rm Law}(X), {\rm Law}(Z)) := \frac{1}{2}\sup _{A\in {\mathcal B}({\mathbb R})} \left| E[{\bf 1}_{A}(X)] - E[{\bf 1}_{A}(Z)] \right| .
\]
Note that
\[
\sup _{A\in {\mathcal B}({\mathbb R})} \left| E[{\bf 1}_{A}(X)] - E[{\bf 1}_{A}(Z)] \right| = \sup_{f\in C_b({\mathbb R}); \| f\| _\infty \leq 1} \left| E[f(X)] - E[f(Z)] \right|
\]
(see \cite[Section 2.12(iv)]{Bo07}).
From this equality and \eqref{eq:MS}, we obtain
\begin{equation}\label{eq:MS2}
d_{\rm TV}({\rm Law}(X), {\rm Law}(Z))
\leq C \left( E\left[ \left| 1- \left\langle D (-L)^{-1} (X-E[X]),  DX \right\rangle _H \right| \right] + |E[X]| \right) .
\end{equation}
We use this estimate in the proofs of the main theorems.

\subsection{Exponential integrability of Brownian functionals}\label{sec:Brownian}

Here, we provide a sufficient condition for the exponential integrability of Brownian functionals.
For $x, a\in {\mathbb R}$ and a Brownian motion $B_t$, denote the stochastic process $x+B_t+at$ by $B_t^{x,a}$.

\begin{prop}\label{prop:integrable} {\rm (\cite[p.~73, {\bf 34}]{BS02}, \cite{SY05})}
Let $a>0$ and let $f$ be a bounded Borel function on ${\mathbb R}$ such that $\int_0^{\infty}f(x)dx<\infty$, 
and let $M=\sup_{x\in {\mathbb R}}f(x)$. 
Then for any $c \in \left[ 0, \frac{a^2}{2M}\right)$, 
\[
E\left[\exp\left(c\int_0^{\infty}f(B_t^{x,a})dt\right)\right]<\infty, \quad x\in {\mathbb R}.
\]
In particular, for any $p\in [0,\infty )$
\[
E\left[ \left| \int_0^{\infty}f(B_t^{x,a})dt \right| ^p \right]<\infty, \quad x\in {\mathbb R}.
\]
\end{prop}

For later use, we prepare an implication of Proposition~\ref{prop:integrable}, as follows.

\begin{prop}\label{prop:integrable2}
Let $a>0$ and let $f$ be a nonnegative non-increasing function on ${\mathbb R}$ such that $\int_0^{\infty}f(x)dx<\infty$, 
and let $M=\sup_{x\in {\mathbb R}}f(x)$. 
\begin{enumerate}
\item For any $c<\frac{a^2}{2M}$ and $x_0\in {\mathbb R}$, 
\[
\sup _{x\in [x_0,\infty )} E \left[\exp\left(c\int_0^{\infty}f(B_t^{x,a})dt\right)\right]<\infty .
\]
\item For any $p\in [0,\infty )$ and $x_0\in {\mathbb R}$, 
\[
\sup _{x\in [x_0,\infty )} E \left[ \left( \int_0^{\infty}f(B_t^{x,a})dt \right) ^p \right]<\infty .
\]
\end{enumerate}
\end{prop}

\begin{proof}
It is sufficient to consider the case that $c>0$.
Since $f$ is non-increasing, for $x\geq x_0$
\begin{align*}
E \left[\exp\left(c\int_0^{\infty}f(B_s^{x,a})dt\right)\right] &= E_{x_0} \left[\exp\left(c\int_0^{\infty}f(x-x_0 + B_t^{x_0,a})dt\right)\right] \\
&\leq E \left[\exp\left(c\int_0^{\infty}f(B_t^{x_0,a})dt\right)\right] .
\end{align*}
Hence, (i) follows from Proposition~\ref{prop:integrable}.
From (i) we have (ii), because for any $\varepsilon >0$ there exists $L \in [1,\infty )$ such that
\[
y^p \leq L e^{\varepsilon y} , \quad y\in [0,\infty ) .
\]
\end{proof}

\subsection{Estimates of martingales with constant drifts}\label{sec:marti+drift}

Here, we consider the properties of the pathwise-unique solution $Y_t^x$ of the SDE 
\begin{equation}\label{eq:SDEY}\left\{\begin{array}{l}
dY_t^x = \sigma (t,Y_t^x) dB_t + b_1 dt, \\
Y_0^x =x \in {\mathbb R}
\end{array}\right.\end{equation}
where $b_1\in {\mathbb R}$ and $\sigma \in C_b^{0,1}([0,\infty )\times {\mathbb R})$ satisfies
\[
\sigma _1 \leq \sigma (t,y) \leq \sigma _2, \quad (t,y) \in [0,\infty )\times {\mathbb R}
\]
with some constants $\sigma _1, \sigma _2 \in (0,\infty )$.
Let
\begin{equation}\label{eq:mtg-y}
M_t^x :=\int_0^t\sigma(s,Y_s^x)dB_s, \quad t\ge 0.
\end{equation}
This is a square-integrable martingale with respect to the filtration generated by $B_t$, and it holds that 
\[
\langle M^x\rangle_t= \int_0^t\sigma(s,Y_s^x)^2ds, \quad t\ge 0,
\]
where $\langle M^x\rangle_t$ is the quadratic variation of $M^x_t$.
Then, \eqref{eq:SDEY} implies
\[
Y_t^x=x+M_t^x+b_1 t, \quad t\ge 0.
\]
Since 
\begin{equation}\label{eq:quad}
\sigma_1^2t\le \langle M^x\rangle_t \le \sigma_2^2 t, \quad t\ge 0,
\end{equation}
there exists a one-dimensional Brownian motion $\{\tilde{B}_t\}_{t\ge 0}$ such that 
\begin{equation}\label{eq:MB}
M_t^x=\tilde{B}_{\langle M^x\rangle_t} ,\quad t\ge 0
\end{equation} 
(see, e.g., \cite[Theorem~4.6 (p.~174)]{KS98}).
This and \eqref{eq:quad} imply that 
\begin{equation}\label{eq:lli-1}
\limsup_{t\rightarrow\infty}\frac{|M_t^x|}{t} \leq \limsup_{t\rightarrow\infty} \frac{\sigma _2^2|\tilde{B}_{\langle M^x\rangle_t}|}{\langle M^x\rangle_t} =0, \quad \mbox{almost surely}.
\end{equation}
Hence we have
\begin{equation}\label{eq:divergence-0}
\lim_{t\rightarrow\infty}\frac{Y_t^x}{t}=b_1, \quad \mbox{almost surely.}
\end{equation}

As extensions of Propositions \ref{prop:integrable} and \ref{prop:integrable2} we prepare the following.

\begin{prop}\label{prop:integrable-mtg} 
Let $a>0$.
\begin{enumerate}
\item[{\rm (i)}] 
Let $f$ be a bounded nonnegative Borel function on ${\mathbb R}$ such that 
$\int_0^{\infty}f(x)dx<\infty$, and let $M=\sup_{x\in {\mathbb R}}f(x)$. 
Then for any $c<\frac{\sigma_1^2 a^2}{2M}$,
\[
E\left[\exp\left(c\int_0^{\infty}f(x+M_t^x+a\langle M^x\rangle_t)dt\right)\right]<\infty, 
\quad x\in {\mathbb R}.
\]
Moreover, if $f$ is nonincreasing, then for any $x_0\in {\mathbb R}$,
\[
\sup_{x\in [x_0,\infty)}E\left[\exp\left(c\int_0^{\infty}f(x+M_t^x+a \langle M^x\rangle_t)dt\right)\right]<\infty.
\]
\item[{\rm (ii)}] 
Let $f$ be a bounded Borel function on ${\mathbb R}$ such that $\int_0^{\infty}|f(x)|dx<\infty$. 
Then for any $p\in [0,\infty)$, 
\[
\sup_{t\ge 0}E\left[\left|\int_0^t f(x+M_t^x+a \langle M^x\rangle_t)dt\right|^p\right]<\infty, \quad x\in {\mathbb R}.
\]
Moreover, if $f$ is nonnegative and nonincreasing, then for any $x_0\in {\mathbb R}$,
\[
\sup_{x\in [x_0,\infty)}E\left[\left(\int_0^{\infty} f(x+M_t^x+a \langle M^x\rangle_t)dt\right)^p\right]<\infty.
\]
\end{enumerate}
\end{prop}

\begin{proof}
Let us prove (i). We may assume that $c>0$. 
From \eqref{eq:MB}, it follows that
\begin{align*}
&\exp\left(c\int_0^{\infty}f(x+M_t^x+a\langle M^x\rangle_t)dt\right)\\
&=\exp\left(c\int_0^{\infty}f(x+\tilde{B}_{\langle M^x\rangle_t}+a\langle M^x\rangle_t)dt\right)\\
&=\exp\left(c\int_0^{\infty}f(x+\tilde{B}_u+a u) \, \sigma\left(\langle M^x\rangle_u^{-1}, Y_{\langle M^x\rangle_u^{-1}}^x\right)^{-2} du\right)\\
&\le \exp\left(\frac{c}{\sigma_1^2}\int_0^{\infty}f(x+\tilde{B}_u+a u)du\right).
\end{align*}
In the second equality above, we used the change of variables formula with $u=\langle M^x\rangle_t$. 
From this inequality we have 
\[
E\left[\exp\left(c\int_0^{\infty}f(x+M_t^x+a\langle M^x\rangle_t)dt\right)\right]
\le E\left[ \exp\left(\frac{c}{\sigma_1^2}\int_0^{\infty}f(x+\tilde{B}_u+a u)du\right)\right].
\]
In view of Proposition~\ref{prop:integrable}, the right-hand side is finite 
for any $c\in [0,\frac{\sigma_1^2 a^2}{2M})$.
Proposition~\ref{prop:integrable2}(i) also implies the desired uniform bound. 

We omit the proof of (ii), because we obtain the assertion similarly by applying Propositions~\ref{prop:integrable} and \ref{prop:integrable2}(ii). 
\end{proof}

By applying Proposition~\ref{prop:integrable-mtg} we prove some estimates of $Y_t^x$ for later use.

\begin{cor}\label{cor:intY} 
\begin{enumerate}
\item[{\rm (i)}] 
Let $f$ be a bounded nonnegative and nonincreasing function on ${\mathbb R}$ such that 
$\int_0^{\infty}f(x)dx<\infty$, and let $M=\sup_{x\in {\mathbb R}}f(x)$. 
Then for any $c\in [0,\frac{\sigma_1^2 b_1 ^2}{2M})$ and $x_0\in {\mathbb R}$,
\[
\sup_{x\in [x_0,\infty)}E\left[\exp\left(c\int_0^{\infty}f(Y_t^x)dt\right)\right]<\infty.
\]
\item[{\rm (ii)}] 
Let $f$ be a bounded nonnegative and nonincreasing function on ${\mathbb R}$ such that $\int_0^{\infty} f(x) dx<\infty$. 
Then, for any $p \in [0,\infty )$ and $x_0\in {\mathbb R}$,
\[
\sup_{x\in [x_0,\infty)}E\left[\left(\int_0^{\infty} f(Y_t^x)dt\right)^p\right]<\infty.
\]
\end{enumerate}
\end{cor}

\begin{proof}
The assertions follow immediately from Proposition \ref{prop:integrable-mtg} and \eqref{eq:quad}.
\end{proof}

\begin{prop}\label{prop:asymY}
\begin{enumerate}

\item For $L\in {\mathbb R}$ there exists a constant $C>0$ depending on $\sigma _1$, $\sigma _2$, $b_1$, $x_0$ and $L$ such that 
\[
\sup _{x\in [x_0,\infty )} P\left( \inf _{s\in [t,\infty )} Y_s^x \leq L\right) \leq C\exp \left( -\frac{b_1^2\sigma _1^2}{16 \sigma _2^4}t \right) , \quad t\geq 1.
\]

\item For $\gamma >0$ there exists a constant $C>0$ depending on $x_0$, $b_1$ and $\sigma _2$ such that 
\[
\sup _{x\in [x_0,\infty )} E\left[ (Y_t^x)^{-\gamma} {\bf 1}_{[1,\infty)} (Y_t^x)\right] \leq Ct^{-\gamma} , \quad t\geq 1.
\]
\end{enumerate}
\end{prop}

\begin{proof}
Let  $t\geq \frac{2 \sigma _2^2 [(L-x_0)\vee 0]}{b_1 \sigma _1^2}$ and $x\geq x_0$. 
From \eqref{eq:SDEY}, \eqref{eq:quad} and \eqref{eq:MB} we have
\begin{align*}
P\left( \inf _{s\in [t,\infty )} Y_s^x \leq L\right) &= P\left( \inf _{s\in [t,\infty )} (M_s^x + b_1 s) \leq L - x \right) \\
&\leq P\left( \inf _{s\in [t,\infty )} \left( \tilde{B}_{\langle M^x\rangle _s} + \frac{b_1}{\sigma _2^2} \langle M^x\rangle _s \right) \leq L - x \right) \\
&\leq P\left( \inf _{s\in [\sigma _1^2 t,\infty )} \left( \tilde{B}_s + \frac{b_1}{\sigma _2^2} s \right) \leq L - x \right) .
\end{align*}
Hence,  by using the scaling property of the Brownian motion, 
\begin{equation}\label{eq:asympY-1}
\begin{split}
P\left( \inf _{s\in [t,\infty )} Y_s^x \leq L\right) 
&\leq \sum _{n=0}^\infty P\left( \inf _{s\in [\sigma _1^2 2^n t, \sigma _1^2 2^{n+1} t] } \left( \tilde{B}_s + \frac{b_1}{\sigma _2^2} s \right) \leq L - x_0 \right)\\
&\leq \sum _{n=0}^\infty P\left( \inf _{s\in [1,2]} \left( \tilde{B}_s + \frac{b_1 \sigma _1 \sqrt{t}}{\sigma _2^2} 2^{\frac{n}{2}}s \right) < \frac{L-x_0}{\sigma _1 \sqrt{t}}2^{-\frac{n}{2}} \right).
\end{split}
\end{equation}
Note that if $L-x_0\ge 0$, then for any $t\geq \frac{2 \sigma _2^2 (L-x_0)}{b_1 \sigma _1^2}$,
\[
\frac{L-x_0}{\sigma_1\sqrt{t}}2^{-\frac{n}{2}}
=\frac{(L-x_0)\sqrt{t}}{\sigma_1 t}2^{-\frac{n}{2}}
\le \frac{b_1\sigma_1\sqrt{t}}{2\sigma_2^2}2^{-\frac{n}{2}}
\le \frac{b_1\sigma_1\sqrt{t}}{2\sigma_2^2}2^{\frac{n}{2}}.
\]
Combining this with the density function of the maximum of the Brownian motion (see, e.g., \cite[Corollary 3.1.8]{MT17}), 
we get 
\begin{align*}
&P\left( \inf _{s\in [1,2]} \left( \tilde{B}_s + \frac{b_1 \sigma _1 \sqrt{t}}{\sigma _2^2} 2^{\frac{n}{2}}s \right) < \frac{L-x_0}{\sigma _1 \sqrt{t}}2^{-\frac{n}{2}} \right)\\
&\le P\left( \inf _{s\in [0,2]} B_s < - \frac{b_1 \sigma _1 \sqrt{t}}{2 \sigma _2^2} 2^{\frac{n}{2}} \right)\\
&= \frac{2}{\sqrt{4\pi}} \int _{\frac{b_1 \sigma _1 \sqrt{t}}{2 \sigma _2^2} 2^{\frac{n}{2}}}^{\infty} e^{-\frac{y^2}{4}} dy \\
&\le  \frac{2}{\sqrt{\pi}} \left( \frac{b_1 \sigma _1 \sqrt{t}}{2 \sigma _2^2} 2^{\frac{n}{2}} \right) ^{-1} 
\exp \left( -\frac{1}{4} \left( \frac{b_1 \sigma _1 \sqrt{t}}{2 \sigma _2^2} 2^{\frac{n}{2}} \right) ^2 \right).
\end{align*}
Then by \eqref{eq:asympY-1},
\begin{align*}
P\left( \inf _{s\in [t,\infty )} Y_s^x \leq L\right) 
&\leq  \sum _{n=0}^\infty \frac{2}{\sqrt{\pi}} \left( \frac{b_1 \sigma _1 \sqrt{t}}{2 \sigma _2^2} 2^{\frac{n}{2}} \right) ^{-1} 
\exp \left( -\frac{1}{4} \left( \frac{b_1 \sigma _1 \sqrt{t}}{2 \sigma _2^2} 2^{\frac{n}{2}} \right) ^2 \right)\\
&\leq \frac{C \sigma _2^2}{b_1 \sigma _1 \sqrt{t}} \exp \left( -\frac{b_1^2\sigma _1^2}{16 \sigma _2^4}t \right),
\end{align*}
where $C$ is a universal constant.
Thus, we obtain (i).

From \eqref{eq:SDEY}, \eqref{eq:quad} and \eqref{eq:MB} we have that for $t> \left( -\frac{2x_0}{b_1}+ 1\right) \vee 0$
\begin{align*}
&E\left[ (Y_t^x)^{-\gamma} {\bf 1}_{[1,\infty)} (Y_t^x)\right] \\
&=  E\left[ (x+M_t^x+b_1t)^{-\gamma} {\bf 1}_{[1,\infty)} (x+M_t^x+b_1t) \right] \\
&\leq \left( x+\frac{b_1}{2}t \right) ^{-\gamma} P\left( M_t^x \geq - \frac{b_1}{2}t \right) + P\left( M_t^x < - \frac{b_1}{2}t \right) \\
&\leq \left( x_0+\frac{b_1}{2}t \right) ^{-\gamma} + \left( \frac{2}{b_1 t} \right) ^{2\gamma} E\left[ |M_t^x|^{2\gamma} \right] .
\end{align*}
Since the Burkholder-Davis-Gundy inequality and \eqref{eq:quad} imply
\[
E\left[ |M_t^x|^{2\gamma} \right] \leq C E\left[ \left( \int_0^t\sigma(s,Y_s^x)^2ds \right) ^{\gamma} \right] \leq C \sigma _2 ^{2\gamma} t^\gamma
\]
where $C$ is a positive constant depending only on $\gamma$, we obtain (ii).
\end{proof}

\begin{prop}\label{prop:asymY2}
\begin{enumerate}
\item There exists a constant $C>0$ depending only on $\sigma _2$ and $b_1$ such that
\[
P\left( \inf _{s\in [0,\infty )} Y_s^x < y \right) \leq C \left((x-y)^{-1} e^{-\frac{1}{2} (x-y)^2} + e^{-\frac{b_1}{\sigma _2^2}(x-y)}\right) , \quad x\geq y .
\]

\item Let $\varepsilon \in (0,b_1)$. Then, there exists a constant $C>0$ depending only on $\sigma _2$, $b_1$ and $\varepsilon$ such that
\[
P\left( Y_t^x \leq y \right) \leq C e^{-\frac{\varepsilon ^2}{2\sigma _2^2}t} ,\quad (b_1 -\varepsilon ) t > y -x.
\]
\end{enumerate}
\end{prop}

\begin{proof}
Similarly to the proof of Proposition~\ref{prop:asymY} we have
\begin{align*}
P\left( \inf _{s\in [0,\infty )} Y_s^x < y \right) &= P\left( \inf _{s\in [0,\infty )} (\tilde{B}_{\langle M^x\rangle _s} +b_1s) \leq y - x \right) \\
&\leq P\left( \inf _{s\in [0,\infty )} \left( \tilde{B}_s + \frac{b_1}{\sigma _2^2} s\right) \leq y - x \right)
\end{align*}
for $x,y \in{\mathbb R}$.
Hence, in the case that $x>y$, by using the scaling property of the Brownian motion and the density function of the maximum of the Brownian motion (see e.g \cite[Corollary 3.1.8]{MT17}), we have
\begin{align*}
&P\left( \inf _{s\in [0,\infty )} Y_s^x < y \right) \\
&\leq P\left( \inf _{s\in [0,1]} \left( \tilde{B}_s + \frac{b_1}{\sigma _2^2} s \right) < y-x\right) + \sum _{n=0}^\infty P\left( \inf _{s\in [1,2]} \left( \tilde{B}_s + \frac{b_1}{\sigma _2^2} 2^{\frac{n}{2}}s \right) < 2^{-\frac{n}{2}}(y-x) \right) \\
&\leq P\left( \inf _{s\in [0,1]} B_s < y-x \right) + \sum _{n=0}^\infty P\left( \inf _{s\in [0,2]} B_s < - \left( \frac{b_1}{\sigma _2^2} 2^{\frac{n}{2}} + 2^{-\frac{n}{2}} (x-y) \right) \right) \\
&\leq \frac{2}{\sqrt{2\pi}} \int _{x-y}^{\infty} e^{-\frac{z^2}{2}} dz + \frac{2}{\sqrt{4\pi}} \sum _{n=0}^\infty \int _{\frac{b_1}{\sigma _2^2}  2^{\frac{n}{2}} + 2^{-\frac{n}{2}} (x-y)}^{\infty} e^{-\frac{y^2}{4}} dy \\
&\leq \frac{2e^{-\frac{1}{2} (x-y)^2} }{\sqrt{2\pi}(x-y)}+ \frac{2}{\sqrt{\pi}} \sum _{n=0}^\infty \left( \frac{b_1}{\sigma _2^2} 2^{\frac{n}{2}} + 2^{-\frac{n}{2}} (x-y) \right) ^{-1} \exp \left( -\frac{1}{4} \left( \frac{b_1}{\sigma _2^2}  2^{\frac{n}{2}} + 2^{-\frac{n}{2}} (x-y) \right) ^2 \right) \\
&\leq C \left((x-y)^{-1} e^{-\frac{1}{2} (x-y)^2} + \frac{\sigma _2^2}{b_1} e^{-\frac{b_1}{\sigma _2^2}(x-y)}\right)
\end{align*}
where $C$ is a universal constant.
Thus, we obtain (i).

Next we show (ii) and consider the case that $(b_1 -\varepsilon ) t > y -x$.
Similarly to the proof of Proposition~\ref{prop:asymY} we also have
\[
P\left( Y_t^x \leq y \right) \leq  \frac{2}{\sqrt{2\pi}} \cdot \frac{\sigma _2 \sqrt{t}}{b_1t-(y-x)} \exp \left( -\frac{(b_1t-(y-x))^2}{2\sigma _2^2 t} \right).
\]
Hence, (ii) holds.
\end{proof}

\section{Case of bounded coefficients}\label{sec:BE}

Consider a one-dimensional SDE
\begin{equation}\label{eq:SDE1}\left\{\begin{array}{l}
dX_t^x = \sigma (t,X_t^x) dB_t + b(t,X_t^x)dt, \\
X_0^x =x \in {\mathbb R},
\end{array}\right.\end{equation}
where $\sigma , b \in C^{0,1}([0,\infty) \times {\mathbb R})$.  
We assume the next conditions on $\sigma$ and $b$.

\begin{ass}\label{ass1}
\begin{itemize}
\item There exist constants $\sigma _1, \sigma _2\in (0,\infty )$ such that
\begin{equation}\label{ass:s1}\tag{3.$\sigma$1}
\sigma _1 \leq \sigma (t,y) \leq \sigma _2, \quad (t,y)\in [0,\infty )\times {\mathbb R}.
\end{equation}

\item There exist constants $\alpha , \sigma _3\in (0, \infty )$ such that
\begin{equation}\label{ass:s2}\tag{3.$\sigma$2}
\sup _{t\in [0,\infty )} \left| \frac{\partial}{\partial y} \sigma (t,y)\right| \leq \sigma _3 (y\vee 1)^{-\alpha -1}, \quad y\in {\mathbb R}.
\end{equation}

\item There exist constants $b _1, b _2\in (0,\infty )$ such that
\begin{equation}\label{ass:b1}\tag{3.$b$1}
b _1 \leq b (t,y) \leq b _2, \quad (t,y)\in [0,\infty )\times {\mathbb R}.
\end{equation}

\item There exist constants $\beta ,b_3 \in (0, \infty )$ such that
\begin{equation}\label{ass:b2}\tag{3.$b$2}
\sup _{t\in [0,\infty )} \left| \frac{\partial}{\partial y} b(t,y)\right| \leq b_3 y^{-\beta -1}, \quad y\geq 1.
\end{equation}

\end{itemize}
\end{ass}

Under the assumption \eqref{ass:s2} for any $y_1,y_2 \in [1,\infty )$ it holds that 
\begin{align*}
\sup _{t\in [0,\infty )} |\sigma (t,y_2) - \sigma (t, y_1)| & \leq \int _{\min\{ y_1,y_2\}}^{\max\{ y_1,y_2\}} \sup _{t\in [0,\infty )} \left| \frac{\partial}{\partial \tilde{y}} \sigma (t,\tilde{y})\right| d\tilde{y} \\
& \leq \frac{\sigma _3 }{\alpha} \min\{ y_1,y_2\} ^{-\alpha} .
\end{align*}
Hence, there exists $\sigma _\infty \in C([0,\infty ); [\sigma _1, \sigma _2])$ such that
\[
\lim _{y\rightarrow \infty} \sup _{t\in [0,\infty )} |\sigma (t,y)-\sigma _\infty (t)|  = 0.
\]
Moreover, it holds that
\begin{equation}\label{eq:ests}
\sup _{t\in [0,\infty )} |\sigma (t,y)-\sigma _\infty (t)| \leq \frac{\sigma _3 }{\alpha} y^{-\alpha}, \quad y\geq 1.
\end{equation}
Similarly, there exists  $b_\infty \in C([0,\infty ); [b_1, b_2])$ such that
\[
\lim _{y\rightarrow \infty} \sup _{t\in [0,\infty )} |b(t,y)-b_\infty (t)| = 0,
\]
and 
\begin{equation}\label{eq:estb}
\sup _{t\in [0,\infty )} |b(t,y)-b_\infty (t)| \leq \frac{b_3}{\beta} y^{-\beta }, \quad y\geq 1.
\end{equation}

On the other hand, from the comparison principle (see \cite[Theorem 1.3]{Ya73}) and \eqref{ass:b2}, we have
\begin{equation}\label{eq:comparison}
Y_t^x \leq X_t^x, \quad t\in [0,\infty ) \quad \mbox{almost surely}
\end{equation}
where $Y$ is the pathwise-unique solution of \eqref{eq:SDEY}.

Now we see limit theorems of $X_t^x$. Let
\[
\overline{\sigma _\infty}(t) := \sqrt{\frac{1}{t} \int _0^t \sigma _\infty (s)^2 ds} \quad \mbox{and} \quad \overline{b _\infty}(t) := \frac{1}{t} \int _0^t b_\infty (s) ds .
\]
Note that, if $\sigma$ and $b$ do not depend on the time parameter $t$, then $\sigma _\infty$ and $b_\infty$, respectively, are independent of $t$.
In this case, we have $\overline{\sigma _\infty} = \sigma _\infty$ and $\overline{b _\infty} = b_\infty$.
By the observations above, we get a version of the law of large numbers as follows.

\begin{thm}\label{thm:main1}
Under Assumption~{\rm \ref{ass1}}, it holds that
\[
\lim _{t\rightarrow \infty} \left( \frac{X_t^x}{t} - \overline{b _\infty}(t) \right) = 0
\]
almost surely for each $x\in {\mathbb R}$.
\end{thm}

\begin{proof}
From \eqref{eq:comparison} and \eqref{eq:divergence-0} we have
\begin{equation}\label{eq:divergence}
\lim_{t\rightarrow\infty}X_t^x=\infty, \quad \mbox{almost surely.}
\end{equation}
Similarly to \eqref{eq:lli-1}, we also obtain
\begin{equation}\label{eq:lli-2}
\lim_{t\rightarrow\infty}\frac{1}{t}\int_0^t\sigma(s,X_s^x)dB_s=0 \quad \mbox{almost surely.}
\end{equation}
From \eqref{eq:estb} and \eqref{eq:divergence}, it follows that
\[
|b(s,X_s^x)-b_\infty (s)| \leq \sup_{t\in [0,\infty)}|b(t,X_s^x)-b_\infty (t)|\rightarrow 0 \quad (s\rightarrow \infty) \quad \mbox{almost surely,}
\]
and thus 
\[
\lim _{t\rightarrow \infty} \frac{1}{t}\int_0^t \left( b(s,X_s^x) - b_\infty (s) \right) ds = 0 \quad \quad \mbox{almost surely.}
\]
Hence, from \eqref{eq:SDE1} and \eqref{eq:lli-2}, we have 
\begin{align*}
\frac{X_t^x}{t} - \overline{b _\infty}(t) &=\frac{x}{t}+\frac{1}{t}\int_0^t\sigma(s,X_s^x)dB_s+\frac{1}{t}\int_0^t \left( b(s,X_s^x) - b_\infty (s) \right) ds \\
&\rightarrow 0 \quad (t\rightarrow \infty) \quad \mbox{almost surely.}
\end{align*}
Thus, we obtain the assertion.
\end{proof}

We also have the central limit theorem as follows.

\begin{thm}\label{thm:CLT}
Under Assumption~{\rm \ref{ass1}} with $\beta >\frac{1}{2}$, it holds that for all $p\in [1,\infty )$ and $t\geq 1$
\begin{align*}
\sup _{x\in [x_0, \infty )} 
E\left[ \left| \frac{X_t^x -x - t \overline{b _\infty}(t) }{\sqrt{t}} - \frac{1}{\sqrt{t}}\int _0^t \sigma _\infty (s) dB_s \right| ^p\right] ^{\frac{1}{p}} \\
\leq C t^{-\gamma}\left(1+\sqrt{\log t}\cdot  {\bf 1}_{\left\{ \frac{1}{2}\right\} }(\alpha )+\log t\cdot {\bf 1}_{\{ 1\}}(\beta )\right),
\end{align*}
where $\gamma = \min \left\{ \alpha , \beta -\frac{1}{2}, \frac{1}{2} \right\}$ and $C$ is a positive constant independent of $t$.
In particular, the law of $\frac{X_t^x - t \overline{b _\infty}(t)}{\overline{\sigma _\infty}(t) \sqrt{t}}$ converges to the standard normal distribution as $t\rightarrow \infty$.
\end{thm}

Before the proof of Theorem \ref{thm:CLT}, we recall the following fact: 
let $(S,{\mathcal G}, \nu)$ be a measure space, 
and let $\Phi = \Phi (\omega ,y)$ be an ${\mathcal F}\otimes {\mathcal G}$-measurable function 
such that $\Phi (\omega ,\cdot )$ is $\nu$-integrable for all $\omega \in \Omega$. 
Then for any $q\in [1,\infty )$, by the triangle inequality, we have
\begin{equation}\label{eq:int-norm}
\left\| \int _S \Phi (\cdot ,y) \nu (dy) \right\| _{L^q(\Omega )} \leq \int _S \left\| \Phi (\cdot ,y) \right\| _{L^q(\Omega )} \nu (dy). 
\end{equation}

\begin{proof}[Proof of Theorem {\rm \ref{thm:CLT}}.]
It is sufficient to show the case that $p\geq 2$.
From \eqref{eq:SDE1}, \eqref{eq:int-norm} and the Burkholder-Davis-Gundy inequality we have
\begin{align*}
&E\left[ \left| \frac{X_t^x -x - t \overline{b _\infty}(t) }{\sqrt{t}} - \frac{1}{\sqrt{t}}\int _0^t \sigma _\infty (s) dB_s \right| ^p \right] ^{\frac{1}{p}}\\
&\leq E\left[ \left| \frac{1}{\sqrt{t}}\left( \int _0^t (\sigma (s,X_s^x) - \sigma _\infty (s)) dB_s \right) \right| ^p \right] ^{\frac{1}{p}} + E\left[ \left| \frac{1}{\sqrt{t}}\left( \int _0^t (b(s,X_s^x) - b_\infty (s)) ds \right) \right| ^p \right] ^{\frac{1}{p}} \\
&\leq \frac{C_p}{\sqrt{t}} E\left[ \left( \int _0^t (\sigma (s,X_s^x) - \sigma _\infty (s))^2 ds \right) ^{\frac{p}{2}}\right] ^{\frac{1}{p}} + \frac{1}{\sqrt{t}}\int _0^t E\left[ \left| b(s,X_s^x) - b_\infty (s) \right| ^p \right] ^{\frac{1}{p}} ds \\
&\leq \frac{C_p}{\sqrt{t}} \left( \int _0^t E\left[ (\sigma (s,X_s^x) - \sigma _\infty (s))^p \right] ^{\frac{2}{p}} ds \right) ^{\frac{1}{2}} + \frac{1}{\sqrt{t}}\int _0^t E\left[ \left| b(s,X_s^x) - b_\infty (s) \right| ^p \right] ^{\frac{1}{p}} ds
\end{align*}
where $C_p$ is a constant depending on $p$.
Then, by \eqref{eq:ests}, \eqref{eq:estb} and \eqref{eq:comparison} we have
\begin{align*}
&E\left[ \left| \frac{X_t^x -x - t \overline{b _\infty}(t) }{\sqrt{t}} - \frac{1}{\sqrt{t}}\int _0^t \sigma _\infty (s) dB_s \right| ^p \right] ^{\frac{1}{p}}\\
&\leq \frac{C_p}{\sqrt{t}} \left( \int _0^t E\left[ (\sigma (s,X_s^x) - \sigma _\infty (s) )^p {\bf 1}_{(-\infty ,1)}(Y_s^x) \right] ^{\frac{2}{p}} ds \right) ^{\frac{1}{2}} \\
&\quad + \frac{C_p}{\sqrt{t}} \left( \int _0^t E\left[ (\sigma (s,X_s^x) - \sigma _\infty (s) )^p {\bf 1}_{[1,\infty )}(Y_s^x) \right] ^{\frac{2}{p}} ds \right) ^{\frac{1}{2}} \\
&\quad + \frac{1}{\sqrt{t}}\int _0^t E\left[ \left| b(s,X_s^x) - b_\infty (s) \right| ^p {\bf 1}_{(-\infty ,1)}(Y_s^x) \right] ^{\frac{1}{p}} ds \\
&\quad + \frac{1}{\sqrt{t}}\int _0^t E\left[ \left| b(s,X_s^x) - b_\infty (s) \right| ^p {\bf 1}_{[1,\infty )}(Y_s^x) \right] ^{\frac{1}{p}} ds \\
&\leq  \frac{2C_p \sigma _2 }{\sqrt{t}} \left( \int _0^t  P(Y_s^x \leq 1) ^{\frac{2}{p}} ds \right) ^{\frac{1}{2}} + \frac{C_p \sigma _3}{\alpha \sqrt{t}} \left( \int _0^t E\left[ \left( Y_s^x\right) ^{-\alpha p} {\bf 1}_{[1,\infty )}(Y_s^x) \right] ^{\frac{2}{p}} ds \right) ^{\frac{1}{2}} \\
&\quad + \frac{2b_2}{\sqrt{t}} \int _0^t  P(Y_s^x \leq 1) ^{\frac{1}{p}} ds + \frac{b_3}{\beta \sqrt{t}} \int _0^t E\left[ \left( Y_s^x\right) ^{-\beta p} {\bf 1}_{[1,\infty )}(Y_s^x) \right] ^{\frac{1}{p}} ds .
\end{align*}
Hence, the first assertion follows from Proposition \ref{prop:asymY}.

Since $\frac{1}{\sqrt{t}}\int _0^t \sigma _\infty (s) dB_s$ has the centered normal distribution with variance
\[
E\left[ \left( \frac{1}{\sqrt{t}}\int _0^t \sigma _\infty (s) dB_s\right) ^2 \right] = \frac{1}{t}\int _0^t \sigma _\infty (s)^2 ds,
\]
the second assertion follows.
\end{proof}

Let us present  the Berry-Esseen bound for the solution $X_t^x$ to the SDE \eqref{eq:SDE1} as follows.

\begin{thm}\label{thm:BE1}
Let $x_0\in {\mathbb R}$ and Assumption~\ref{ass1} with $\alpha >\frac{1}{2}$ and $\beta >1$ hold, and consider the solution $X_t^x$ of \eqref{eq:SDE1} for $x\in [x_0,\infty )$.
Furthermore, assume that there exists $q\in (1,\infty )$ satisfying
\begin{equation}\label{ass:sb}
\sup _{(t,y)\in [0,\infty )\times {\mathbb R}} \left( \frac{q(q+1)}{q-1} \left( \frac{\partial \sigma}{\partial y} (t,y) \right) ^2 +2q \frac{\partial b}{\partial y}(t,y) \right) < \frac{\sigma _1^2 b_1^2}{2}.
\end{equation}
Then, there exists a positive constant $C$ depending only on
\[
\sigma _1, \ \sigma _2, \ \sigma _3, \ b_1, \ b_2 , \ b_3 , \ \alpha , \ \beta ,\ q\ \mbox{and}\ x_0
\]
such that 
\[
\sup _{x\in [x_0, \infty )} d_{\rm  TV} \left( {\rm Law} \left(\frac{X_t^x -x - t\overline{b_\infty}(t)}{\overline{\sigma _\infty}(t) \sqrt{t}} \right) , {\rm Law}(Z)\right) \leq \frac{C}{\sqrt{t}}, \quad t\geq 1,
\]
where $Z$ is a random variable with the standard normal distribution.
In particular, it holds that
\[
\sup _{x\in [x_0, \infty )} \sup _{y\in {\mathbb R}} \left| P\left( \frac{X_t^x - x - t\overline{b_\infty}(t)}{\overline{\sigma _\infty}(t) \sqrt{t}} \geq y \right) - P(Z\geq y ) \right| \leq \frac{C}{\sqrt{t}}, \quad t\geq 1 .
\]
\end{thm}

For the proof of Theorem~\ref{thm:BE1}, we rely on the Malliavin calculus, in particular the Malliavin-Stein method.
Below, for the brevity of notations, 
we denote the derivatives of $\sigma (s,y)$ and $b(s,y)$ with respect to the spatial variable $y$ by $\partial _y \sigma (s,y)$ and $\partial _y b(s,y)$, respectively.
Throughout this section, 
$C$ is a positive constant depending on $\sigma _1$, $\sigma _2$, $\sigma _3$, $b_1$, $b_2$, $b_3$, $\alpha$, $\beta$, $q$ and $x_0$, which can be different from line to line.
When the constant $C$ depends on another parameter, for example $p$, then we denote it by a subscript, like $C_p$.

\subsection{Estimates of some functionals of $X_t^x$}

We prepare some estimates for the proof of Theorem~\ref{thm:BE1}.

\begin{lem}\label{lem:expest}
Let Assumption~{\rm \ref{ass1}} hold.
Then, for any $\tilde{c}_1, \tilde{c}_2\geq 0$ satisfying
\[
M := \sup _{(t,y)\in [0,\infty )\times {\mathbb R}} \left( \tilde{c}_1 [\partial _y \sigma (t,y)]^2 + \tilde{c}_2 \partial _y b(t,y) \right) < \frac{\sigma _1^2 b_1^2}{2},
\]
it holds that
\[
\sup _{x\in [x_0,\infty )} \sup _{s,t\in [0,\infty ); s<t} E\left[ \exp \left( \int _s^t \left( \tilde{c}_1 [\partial _y \sigma (v, X_v^x) ]^2 + \tilde{c}_2 \partial _y b(v, X_v^x) \right) dv \right) \right] \leq C.
\]
\end{lem}

\begin{proof}
Let $p>1$ satisfy $pM < \frac{\sigma _1^2 b_1^2}{2}$, and let $q:= \frac{p}{p-1}$. 
Take $\delta>0$ so small that $q\delta<\frac{\sigma_1^2b_1^2}{2}$. 
Let $\gamma := \min \{ 2\alpha , \beta\}$. 
Then it follows from \eqref{ass:s2} and \eqref{ass:b2} that, for any $\varepsilon\in (0,\gamma)$, 
there exists $L\in [1,\infty )$ depending on $\sigma _1$, $\sigma _3$, $b_1$ and $b_3$ such that 
\[
\sup _{t\in [0,\infty )} \left( \tilde{c}_1 [\partial _y \sigma (t,y)]^2 + \tilde{c}_2 \partial _y b(t,y) \right) 
\leq \delta y ^{-(\gamma +1-\varepsilon)}, \quad y \in [L,\infty ).
\]
Hence, by \eqref{eq:comparison} and H\"older's inequality, we have
\begin{align*}
&E\left[ \exp \left( \int _s^t \left( \tilde{c}_1 [\partial _y \sigma (v, X_v^x) ]^2 + \tilde{c}_2 \partial _y b(v, X_v^x) \right) dv \right) \right] \\
&\leq E\left[ \exp \left( M \int _s^t {\mathbf 1}_{(-\infty ,L]}(Y_v^x) dv \right) 
\exp \left( \delta \int _s^t \left( X_v^x \right) ^{-(\gamma +1-\varepsilon)} {\mathbf 1}_{(L, \infty )}(Y_v^x) dv \right)\right] \\
&\leq E\left[ \exp \left( M \int _s^t {\mathbf 1}_{(-\infty ,L]}(Y_v^x) dv \right) \exp \left( \delta\int _s^t \left( Y_v^x \right) ^{-(\gamma +1-\varepsilon)} {\mathbf 1}_{(L, \infty )}(Y_v^x) dv \right) \right] \\
&\leq E\left[ \exp \left( pM \int _s^t {\mathbf 1}_{(-\infty ,L]}(Y_v^x) dv \right) \right] ^{\frac{1}{p}} E\left[ \exp \left( \delta q \int _s^t \left( 1+\max\{ Y_v^x, 0\} \right) ^{-(\gamma +1-\varepsilon)} dv \right) \right] ^{\frac{1}{q}} .
\end{align*}
Since $\delta q<\frac{\sigma_1^2 b_1^2}{2}$ and $0<\varepsilon<\gamma$, 
we can apply Corollary~\ref{cor:intY}  to the right-hand side and obtain the assertion. 
\end{proof}

\begin{lem}\label{lem:estLpIb}
Let Assumption~{\rm \ref{ass1}} hold.
\begin{enumerate}
\item[\rm (i)]
Suppose that \eqref{ass:s1} is satisfied with some $\alpha >\frac{1}{2}$. 
Then for any $p\in [1,\infty )$, 
\[
\sup _{x\in [x_0, \infty )} \sup _{t\in [0,\infty )} E\left[ \left| \int _0^t \left( \sigma (s,X_s^x) - \sigma _\infty (s) \right) dB_s \right| ^p \right] \leq C_p .
\]

\item[\rm (ii)]
Suppose that \eqref{ass:b1} is satisfied with some $\beta >1$. 
Then for any $p\in [1,\infty )$, 
\[
\sup _{x\in [x_0, \infty )} \sup _{t\in [0,\infty )} E\left[ \left| \int _0^t \left( b(s,X_s^x) - b_\infty (s) \right) ds \right| ^p \right] \leq C_p .
\]
\end{enumerate}
\end{lem}

\begin{proof}
Note that \eqref{eq:ests} holds with $\alpha>\frac{1}{2}$.
From the Burkholder-Davis-Gundy inequality, \eqref{ass:s1}, \eqref{eq:ests} and \eqref{eq:comparison} we have  
\begin{align*}
&E\left[ \left| \int _0^t \left( \sigma (s,X_s^x) - \sigma _\infty (s) \right) dB_s \right| ^p \right] \\
&\leq C_p E\left[ \left( \int _0^t \left( \sigma (s,X_s^x) - \sigma _\infty (s) \right) ^2 ds \right) ^{\frac{p}{2}} \right] \\
&\leq C_p \left( \frac{\sigma _3^p}{\alpha ^p} E\left[ \left( \int _0^t (X_s^x)^{-2\alpha} {\bf 1}_{[1,\infty )}(Y_s^x) ds \right) ^{\frac{p}{2}} \right]  + (2\sigma _2)^p E\left[ \left( \int _0^t {\bf 1}_{(-\infty , 1)}(Y_s^x) ds \right) ^{\frac{p}{2}} \right] \right) \\
&\leq C_p  \left( \frac{\sigma _3^p}{\alpha ^p} E\left[ \left( \int _0^t (Y_s^x)^{-2\alpha} {\bf 1}_{[1,\infty )}(Y_s^x) ds \right) ^{\frac{p}{2}} \right]  + (2\sigma _2)^p E\left[ \left( \int _0^t {\bf 1}_{(-\infty , 1)}(Y_s^x) ds \right) ^{\frac{p}{2}} \right] \right) .
\end{align*}
Therefore, (i) follows from Corollary~\ref{cor:intY}.

Since the proof of (ii) is similar by using \eqref{eq:int-norm}, we omit it. 
\end{proof}

\subsection{Estimates of Malliavin derivatives}\label{subsec:MCbdd}

In this section, we apply the Malliavin calculus to the SDE \eqref{eq:SDE1}. 
Let
\begin{align*}
W&:= \left\{ w\in C([0,\infty); {\mathbb R});\ w(0)=0, \ \lim _{t\rightarrow \infty}\frac{w(t)}{t}=0 \right\}, \\
\| w\| _W&:= \sup _{t\geq 0}\frac{|w(t)|}{1+t} , \quad w\in W.
\end{align*}
Then, $(W, \| \cdot \| _W)$ is a separable Banach space and the Wiener measure $\mu$ is defined on $W$ (see \cite[Section 1.3]{DS89}).
Moreover, if we set
\begin{align*}
&H:= \left\{ h\in W; \ \mbox{there exists}\ \dot{h}\in L^2([0,\infty), dt)\ \mbox{such that}\ h(t) = \int _0^t \dot{h}(s)ds\ (t\geq 0) \right\}, \\
&\| h\| _H:= \| \dot{h}\| _{L^2([0,\infty), dt)}, \quad h\in H \ (\dot{h} \ \mbox{is the one appeared in the definition of}\ H),
\end{align*}
then $H$ is the Cameron-Martin space associated with $(W,\mu)$ (see \cite[Section 8.1]{St11}).

The pathwise uniqueness of the SDE \eqref{eq:SDE1} implies that the solution $X^x$ is regarded as a functional of the driving Brownian motion $B_t$ and the Malliavin calculus is applicable under the Wiener space generated by $B_t$.
For the unity of notations, we let $\Omega = W$ and $P =\mu$, and apply the Malliavin calculus to the functionals of $B_t$ with the Cameron-Martin space $H$.

By using Lemma~\ref{lem:expest}, we prove the next two propositions, which provide the bounds on the expectation related to the Malliavin derivative $DX^x$ of $X^x$.  

\begin{prop}\label{prop:estbDX}
Let Assumption~{\rm \ref{ass1}} and \eqref{ass:sb} hold. 
Let $x_0 \in {\mathbb R}$ and let
\[
S_t^x:= \int _0^t (\sigma (s,X_s^x) -\sigma _\infty (s))dB_s + \int _0^t (b(s,X_s^x ) -b_\infty (s)) ds.
\]
\begin{enumerate}
\item[{\rm (i)}] It holds that for any $t\in [0,1]$, 
\[
\sup _{x\in [x_0,\infty )} E \left[ \left| D S_t^x \right| _H^2 \right] \leq Ct.
\]
\item[{\rm (ii)}] It holds that for $t\ge 1$, 
\[
\sup _{x\in [x_0,\infty )} E \left[ \left| D S_t^x \right| _H^2 \right]
\le 
C\times 
\begin{dcases}
t^{1-2(\alpha \wedge \beta)} & \left( 0< \alpha \wedge \beta < \frac{1}{2}\right), \\
1+\log t & \left( \alpha \wedge \beta = \frac{1}{2} \right), \\
1 & \left( \alpha \wedge \beta > \frac{1}{2}\right) .
\end{dcases}
\]
\end{enumerate}
\end{prop}

\begin{proof}
We first compute $\left| D S_t^x \right| _H^2 $.
For $h\in H$, since
\[ \left\{ \begin{array}{rl}
\displaystyle d\langle DX_t^x , h \rangle _H \!\!\!\!\! &\displaystyle = \partial _y \sigma (t,X_t^x) \langle DX_t^x , h \rangle _H dB_t + \sigma (t,X_t^x) dh(t) \\
\displaystyle &\displaystyle \quad \hspace{4cm} + \partial _y b(t,X_t^x) \langle DX_t^x , h \rangle _H dt, \\
\displaystyle DX_0^x \!\!\!\!\! &\displaystyle =0
\end{array}\right.
\]
(see \cite[Theorem 2.2.1]{Nu06} or \cite[Proposition 6.1]{Shi04}), we have
\begin{equation}\label{eq:DXh}
\langle DX_t^x , h \rangle _H = e^{Z_t^x} \int _0^t e^{-Z_s^x} \sigma (s,X_s^x) \dot{h}(s)ds
\end{equation}
where
\begin{equation}\label{eq:z-def}
Z_t^x := \int _0^t \partial _y \sigma (s,X_s^x)dB_s + \int _0^t \partial _y b(s,X_s^x)ds -\frac{1}{2} \int _0^t [\partial _y \sigma (s,X_s^x)]^2 ds
\end{equation}
(see \cite[Section 5.6.C]{KS98}).
On the other hand, since \eqref{eq:SDE1} implies that
\[
S_t^x = X_t^x - \int _0^t \sigma _\infty (s)dB_s ,
\]
it holds that
\[
\langle DS_t^x , h \rangle _H = \langle DX_t^x , h \rangle _H - \int _0^t \sigma _\infty (s) dh(s), \quad h\in H.
\]
Hence, we have
\begin{equation}\label{eq:DSh}
\begin{split}
\langle DS_t^x , h \rangle _H 
&= e^{Z_t^x} \int _0^t e^{-Z_s^x} \sigma (s,X_s^x)\dot{h}(s)ds - \int _0^t \sigma _\infty (s) \dot{h}(s) ds \\
&=\int _0^t \left(e^{Z_t^x-Z_s^x} \sigma (s,X_s^x) - \sigma _\infty (s) \right) \dot{h}(s) ds, \quad h\in H. 
\end{split}
\end{equation}
From this equality, if $\{ h_i\}$ is a complete orthonormal system of $H$, 
then 
\begin{align*}
\left| D S_t^x \right| _H^2
& = \sum _{i=1}^\infty \left| \langle DS_t^x , h_i \rangle _H \right|  ^2 \\
& = \sum _{i=1}^\infty \left| \int _0^\infty \left( e^{Z_t^x -Z_s^x} \sigma (s,X_s^x) - \sigma _\infty (s) \right) {\bf 1}_{[0,t]}(s) \dot{h}_i(s) ds \right| ^2.
\end{align*}
Moreover, we have by the Parseval equality,  
\begin{equation*}
\begin{split}
\left| D S_t^x \right| _H^2
& = \int _0^\infty \left( e^{Z_t^x -Z_s^x} \sigma (s,X_s^x) - \sigma _\infty (s) \right) ^2 {\bf 1}_{[0,t]}(s) ds \\
&= \int _0^t \left( e^{Z_t^x -Z_s^x} \sigma (s,X_s^x) - \sigma _\infty (s) \right) ^2 ds
\end{split}
\end{equation*}
and so 
\begin{equation}\label{eq:DSh-norm}
E\left[\left| D S_t^x \right| _H^2\right] = \int _0^t E \left[ \left( e^{Z_t^x -Z_s^x} \sigma (s,X_s^x) - \sigma _\infty (s) \right) ^2 \right] ds.
\end{equation}

Now we prove (i).
Since
\begin{equation*}\begin{array}{l}
\displaystyle \left( e^{Z_t^x -Z_s^x} \sigma (s,X_s^x) - \sigma _\infty (s) \right) ^2 \\
\displaystyle \leq 2 \left( \sigma (s,X_s^x) - \sigma _\infty (s) \right) ^2 e^{2(Z_t^x-Z_s^x)} + 2\sigma _\infty (s) ^2 \left( e^{Z_t^x-Z_s^x} -1 \right) ^2, 
\end{array}\end{equation*}
it follows by \eqref{eq:DSh-norm} that 
\begin{equation}\label{eq:estDS2}
\begin{split}
&E\left[\left| D S_t^x \right| _H^2\right]  \\
& \leq 2\int_0^t E \left[ \left( \sigma (s,X_s^x) - \sigma _\infty (s) \right) ^2 e^{2(Z_t^x-Z_s^x)} \right] ds 
 + 2\sigma_2^2 \int _0^t E\left[ \left( e^{Z_t^x-Z_s^x} -1 \right) ^2 \right] ds.
\end{split}
\end{equation}
In view of \eqref{ass:sb}, we are able to choose $p,q\in (1,\infty )$ so that $p^{-1}+q^{-1}<1$ and
\[
\sup _{(t,y)\in [0,\infty )\times {\mathbb R}} \left( q(2p-1) [\partial _y \sigma (t,y)]^2 +2q \partial _y b(t,y) \right) < \frac{\sigma _1^2 b_1^2}{2} .
\]
Let $r:= \frac{pq}{pq-p-q} \in (1,\infty)$ so that $p^{-1}+q^{-1}+r^{-1}=1$.
Since H\"older's inequality implies that
\begin{align*}
&E\left[ \left( \sigma (s,X_s^x) - \sigma _\infty (s)\right) ^2 e^{2(Z_t^x-Z_s^x)} \right] \\
&= E\left[ \left( \sigma (s,X_s^x) - \sigma _\infty (s) \right) ^2 \exp \left( 2 \int _s^t \partial _y \sigma (u,X_u^x)dB_u - 2p \int _s^t [\partial _y \sigma (u,X_u^x)]^2 du \right) \right. \\
&\quad \hspace{4cm} \left. \times \exp \left( (2p -1)\int _s^t [\partial _y \sigma (s,X_s^x)]^2 ds + 2 \int _s^t \partial _y b(u,X_u^x)du \right) \right] \\
&\leq E\left[ \exp \left( 2p \int _s^t \partial _y \sigma (u,X_u^x)dB_u - 2p^2 \int _s^t [\partial _y \sigma (u,X_u^x)]^2 du \right) \right] ^{\frac{1}{p}}\\
&\quad \times E\left[ \exp \left( q (2p -1)\int _s^t [\partial _y \sigma (u,X_u^x)]^2 du + 2 q \int _s^t \partial _y b(u,X_u^x)du \right) \right] ^{\frac{1}{q}} \\
&\quad \times E\left[ \left| \sigma (s,X_s^x) - \sigma _\infty (s) \right| ^{2r} \right] ^{\frac{1}{r}} ,
\end{align*}
by the Girsanov transform and Lemma~\ref{lem:expest} we have
\begin{equation}\label{eq:DSh-est-}
\sup _{x\in [x_0,\infty )} E\left[ \left( \sigma (s,X_s^x) - \sigma _\infty (s) \right) ^2 e^{2(Z_t^x-Z_s^x)} \right] \leq C E\left[ \left| \sigma (s,X_s^x) - \sigma _\infty (s) \right| ^{2r} \right] ^{\frac{1}{r}} .
\end{equation}
Similarly, we have
\[
\sup _{x\in [x_0,\infty )} E\left[ e^{2(Z_t^x-Z_s^x)} \right] \leq C .
\]
This equality, \eqref{ass:s1}, \eqref{eq:DSh-norm}, \eqref{eq:estDS2} and \eqref{eq:DSh-est-} yield (i).

Next we prove (ii).
Note that similarly to \eqref{eq:DSh-est-} it holds that
\begin{equation}\label{eq:DSh-est-3}
\begin{split}
\sup _{x\in [x_0,\infty )} E\left[ e^{2(Z_t^x-Z_s^x)} {\bf 1}_{(-\infty ,1)}\left( \inf _{v\in [s,t]} Y_v^x\right) \right] 
&\leq C E\left[ {\bf 1}_{(-\infty ,1)}\left( \inf _{v\in [s,t]} Y_v^x\right) \right] ^{\frac{1}{r}} \\
&=C P\left( \inf_{v\in [s,t]} Y_v^x <1 \right)^{\frac{1}{r}}.
\end{split}
\end{equation}
Since It\^o's formula implies that
\[
e^{Z_t^x} - e^{Z_s^x} = \int _s^t e^{Z_u^x} \partial _y \sigma (u,X_u^x) dB_u + \int _s^t e^{Z_u^x} \partial _y b(u,X_u^x)du ,
\]
from \eqref{eq:int-norm} and \eqref{eq:DSh-est-3} it holds that
\begin{align*}
&E\left[ \left( e^{Z_t^x-Z_s^x} -1 \right) ^2\right] \\
&= E\left[ \left( e^{Z_t^x-Z_s^x} -1 \right) ^2 {\bf 1}_{[1,\infty )}\left( \inf _{v\in [s,t]} Y_v^x\right) \right] + E\left[ \left( e^{Z_t^x-Z_s^x} -1 \right) ^2 {\bf 1}_{(-\infty ,1)}\left( \inf _{v\in [s,t]} Y_v^x\right)\right] \\
&\leq 2 E\left[ \left( \int _s^t e^{Z_u^x -Z_s^x} \partial _y \sigma (u,X_u^x) dB_u \right) ^2 {\bf 1}_{[1,\infty )}\left( \inf _{v\in [s,t]} Y_v^x\right) \right] \\
&\quad + 2 E\left[ \left( \int _s^t e^{Z_u^x -Z_s^x} \partial _y b(u,X_u^x)du \right) ^2 {\bf 1}_{[1,\infty )}\left( \inf _{v\in [s,t]} Y_v^x\right) \right] + C P\left( \inf _{v\in [s,t]} Y_v^x <1 \right) ^{\frac{1}{r}}\\
&\leq 2 \int _s^t E\left[ e^{2(Z_u^x -Z_s^x)} |\partial _y \sigma (u,X_u^x) |^2 \right] du \\
&\quad + 2 \left( \int _s^t E\left[ e^{2(Z_u^x -Z_s^x)} |\partial _y b(u,X_u^x)|^2 {\bf 1}_{[1,\infty )}\left( Y_u^x\right) \right] ^{\frac{1}{2}} du \right) ^2 + C P\left( \inf _{v\in [s,t]} Y_v^x <1 \right) ^{\frac{1}{r}}.
\end{align*}
Similarly to the proof of \eqref{eq:DSh-est-}, we also have
\begin{align*}
\sup _{x\in [x_0,\infty )} E\left[ e^{2(Z_u^x -Z_s^x)} |\partial _y \sigma (u,X_u^x) |^2 \right] &\leq C E\left[ \left| \partial _y \sigma (u,X_u^x) \right| ^{2r} \right] ^{\frac{1}{r}} ,\\
\sup _{x\in [x_0,\infty )} E\left[ e^{2(Z_u^x -Z_s^x)} |\partial _y b(u,X_u^x)|^2 {\bf 1}_{[1,\infty )}\left( Y_u^x\right) \right] &\leq C E\left[ \left| \partial _y b(u,X_u^x) \right| ^{2r} {\bf 1}_{[1,\infty )}\left( Y_u^x\right) \right] ^{\frac{1}{r}} .
\end{align*}
Thus, we have
\begin{equation}\label{eq:DSh-est-2}\begin{array}{l}
\displaystyle \sup _{x\in [x_0,\infty )} E\left[ \left( e^{Z_t^x-Z_s^x} -1 \right) ^2\right] \\
\displaystyle \leq C \left( P\left( \inf _{v\in [s,t]} Y_v^x <1 \right) ^{\frac{1}{r}} + \int _s^t E\left[ \left| \partial _y \sigma (u,X_u^x) \right| ^{2r} \right] ^{\frac{1}{r}} du \right. \\
\displaystyle \quad \hspace{2cm} \left. + \left( \int _s^t E\left[ \left| \partial _y b(u,X_u^x) \right| ^{2r} {\bf 1}_{[1,\infty )}\left( Y_u^x\right) \right] ^{\frac{1}{2r}} du \right) ^2 \right) .
\end{array}\end{equation}
This equality, \eqref{ass:s1}, \eqref{eq:DSh-norm}, \eqref{eq:estDS2}, \eqref{eq:DSh-est-} and Proposition~\ref{prop:asymY}(i) yield
\begin{equation}\label{eq:DSh-est}\begin{array}{l}
\displaystyle \sup _{x\in [x_0,\infty )} E\left[ \left| D S_t^x \right| _H^2 \right] \\
%
%
\displaystyle \leq C \left( 1+ \int _0^t E\left[ \left| \sigma (s,X_s^x) - \sigma _\infty (s) \right| ^{2r} \right] ^{\frac{1}{r}} ds + \int _0^t \int _s^t E\left[ \left| \partial _y \sigma (u,X_u^x) \right| ^{2r} \right] ^{\frac{1}{r}} du ds\right. \\
\displaystyle \quad \hspace{3.5cm} \left. + \int _0^t \left( \int _s^t E\left[ \left| \partial _y b(u,X_u^x) \right| ^{2r} {\bf 1}_{[1,\infty )}\left( Y_u^x\right) \right] ^{\frac{1}{2r}} du \right) ^2 ds \right) .
\end{array}\end{equation}
From \eqref{eq:ests} and \eqref{eq:comparison} we have
\begin{align*}
\left| \sigma (s,X_s^x) - \sigma _\infty (s) \right| ^{2r} &= \left| \sigma (s,X_s^x) - \sigma _\infty (s) \right| ^{2r} \left( {\bf 1}_{(-\infty ,1]}(Y_s^x) + {\bf 1}_{(1,\infty )}(Y_s^x) \right) \\
&\leq (2 \sigma _2)^{2r} {\bf 1}_{(-\infty ,1]}(Y_s^x) + \left( \frac{\sigma _3}{\alpha}\right) ^{2r} (X_s^x)^{-2\alpha r} {\bf 1}_{(1,\infty )}(Y_s^x)\\
&\leq (2 \sigma _2)^{2r} {\bf 1}_{(-\infty ,1]}(Y_s^x) + \left( \frac{\sigma _3}{\alpha}\right) ^{2r} (Y_s^x)^{-2\alpha r} {\bf 1}_{(1,\infty )}(Y_s^x) .
\end{align*}
Hence, Proposition~\ref{prop:asymY} yields
\begin{equation}\label{eq:asys}
E\left[ \left| \sigma (s,X_s^x) - \sigma _\infty (s) \right| ^{2r} \right] ^{\frac{1}{r}} \leq C\exp \left( -\frac{b_1^2 \sigma _1^2}{16 \sigma _2^4 r} s\right) + C s^{-2\alpha}. 
\end{equation}
By \eqref{ass:s2} and \eqref{eq:comparison} we have
\begin{align*}
\left| \partial _y \sigma (u,X_u^x) \right| ^{2r} &= \left| \partial _y \sigma (u,X_u^x) \right| ^{2r} \left( {\bf 1}_{(-\infty ,1]}(Y_u^x) + {\bf 1}_{(1,\infty )}(Y_u^x) \right) \\
&\leq \sigma _3^{2r} {\bf 1}_{(-\infty ,1]}(Y_u^x) + \sigma _3 ^{2r} (X_u^x)^{-2 (\alpha +1) r} {\bf 1}_{(1,\infty )}(Y_u^x)\\
&\leq \sigma _3^{2r} {\bf 1}_{(-\infty ,1]}(Y_u^x) + \sigma _3 ^{2r} (Y_u^x)^{-2 (\alpha +1) r} {\bf 1}_{(1,\infty )}(Y_u^x) .
\end{align*}
Hence, Proposition~\ref{prop:asymY} yields
\begin{equation}\label{eq:asys2}
\int _s^t E\left[ \left| \partial _y \sigma (u,X_u^x) \right| ^{2r} \right] ^{\frac{1}{r}} du \leq C\exp \left( -\frac{b_1^2 \sigma _1^2}{16 \sigma _2^4 r} s\right) + Cs^{-2\alpha -1}.
\end{equation}
Similarly, from \eqref{ass:b2}, \eqref{eq:comparison} and Proposition~\ref{prop:asymY} we have
\begin{equation}\label{eq:asyb}
\int _s^t E\left[ \left| \partial _y b(u,X_u^x) \right| ^{2r} {\bf 1}_{[1,\infty )}\left( Y_u^x\right) \right] ^{\frac{1}{2r}} du \leq Cs^{-\beta}.
\end{equation}
From \eqref{eq:DSh-est}, \eqref{eq:asys}, \eqref{eq:asys2} and \eqref{eq:asyb}, we have (ii).
\end{proof}

\begin{prop}\label{prop:estbDXh}
Let Assumption~{\rm \ref{ass1}} and \eqref{ass:sb} hold. 
Let $x_0 \in {\mathbb R}$, and define $S_t^x$ as in Proposition~{\rm \ref{prop:estbDX}}.
\begin{enumerate}
\item[{\rm (i)}] 
It holds that
\[
\sup _{x\in [x_0,\infty )} E \left[ \left| \left\langle \int _0^\cdot \sigma _\infty (s) {\mathbf 1}_{[0,t]}(s) ds , DS_t^x \right\rangle _H \right| \right] \leq C t, \quad t\in [0,1]
\]
where $\int _0^\cdot \sigma _\infty (s) {\mathbf 1}_{[0,t]}(s) ds\in H$ is a function given by
\[
u\mapsto \int _0^u \sigma _\infty (s) {\mathbf 1}_{[0,t]}(s) ds.
\]

\item[{\rm (ii)}]
It holds that for $t\geq 1$, 
\[
\sup _{x\in [x_0,\infty )} E \left[ \left| \left\langle \int _0^\cdot \sigma _\infty (s) {\mathbf 1}_{[0,t]}(s) ds , DS_t^x \right\rangle _H \right| \right] \leq C \times 
\begin{dcases}
t^{1-(\alpha \wedge \beta)} & \left( 0< \alpha \wedge \beta < 1\right), \\
1+\log t & \left( \alpha \wedge \beta = 1 \right), \\
1 & \left( \alpha \wedge \beta > 1\right) .
\end{dcases}
\]
\end{enumerate}
\end{prop}

\begin{proof}
From \eqref{eq:DSh} we have
\begin{align*}
&\left\langle \int _0^\cdot \sigma _\infty (s) {\mathbf 1}_{[0,t]}(s) ds, DS_t^x \right\rangle _H \\
&= \int _0^\infty \left( e^{Z_t^x-Z_s^x} \sigma (s,X_s^x) - \sigma _\infty (s)\right) \sigma _\infty (s) {\mathbf 1}_{[0,t]}(s) ds \\
& = \int _0^t \sigma _\infty (s) \left( \sigma (s,X_s^x) - \sigma _\infty (s) \right)e^{Z_t-Z_s} ds + \int _0^t \sigma _\infty (s)^2 \left( e^{Z_t-Z_s} - 1 \right) ds
\end{align*}
where $Z_t^x$ is the same as in \eqref{eq:z-def}.
Hence, \eqref{eq:DSh-est-} and \eqref{eq:DSh-est-2} yield
\begin{align*}
&\sup _{x\in [x_0,\infty )} E \left[ \left| \left\langle \int _0^\cdot \sigma _\infty (s) {\mathbf 1}_{[0,t]}(s) ds, DS_t^x \right\rangle _H \right| \right] \\
&\leq \sigma _2 \int _0^t E\left[ \left( \sigma (s,X_s^x) - \sigma _\infty (s)\right) ^2 e^{2(Z_t^x-Z_s^x)}\right] ^{\frac{1}{2}}ds + \sigma _2^2 \int _0^t E\left[ \left( e^{Z_t-Z_s} - 1 \right) ^2 \right] ^{\frac{1}{2}}ds\\
&\leq C \left( \int _0^t E\left[ \left| \sigma (s,X_s^x) - \sigma _\infty (s) \right| ^{2r} \right] ^{\frac{1}{2r}} ds + \int _0^t P\left( \inf _{v\in [s,t]} Y_v^x <1 \right) ^{\frac{1}{2r}} ds \right. \\
& \quad \hspace{1cm} + \int _0^t \left( \int _s^t E\left[ \left| \partial _y \sigma (u,X_u^x) \right| ^{2r} \right] ^{\frac{1}{r}} du \right) ^{\frac{1}{2}} ds \\
& \quad \hspace{5cm} \left.  + \int _0^t \int _s^t E\left[ \left| \partial _y b(u,X_u^x) \right| ^{2r} {\bf 1}_{[1,\infty )}\left( Y_u^x\right) \right] ^{\frac{1}{2r}} du ds \right) 
\end{align*}
for some $r\in (1,\infty )$.
Therefore, (i) and (ii) follow from Proposition~\ref{prop:asymY}(i), \eqref{eq:asys}, \eqref{eq:asys2} and \eqref{eq:asyb}.
\end{proof}

\subsection{Proof of Theorem~{\rm \ref{thm:BE1}}}

In this section, we prove Theorem~\ref{thm:BE1}.

\begin{proof}[{Proof of Theorem~{\rm \ref{thm:BE1}}}]
Let 
\[
F_t^x := \frac{X_t^x -x - t\overline{b_\infty}(t)}{\overline{\sigma _\infty}(t)\sqrt{t}}, \quad t>0.
\]
In view of \eqref{eq:MS2}, it is sufficient to show that
\begin{equation}\label{eq:BE1-01}
E\left[ \left| 1- \left\langle D (-L)^{-1} \left( F_t^x - E[F_t^x] \right), DF_t^x \right\rangle _H \right| \right] + |E[F_t^x]| \leq \frac{C}{\sqrt{t}}, \quad t\geq 1.
\end{equation}
From \eqref{eq:SDE1} we have
\begin{align*}
F_t^x &= \frac{1}{\overline{\sigma _\infty}(t) \sqrt{t}} \left( \int _0^t \sigma (s,X_s^x) dB_s + \int _0^t (b(s,X_s^x) -b_\infty (s)) ds \right) \\
&= \frac{1}{\overline{\sigma _\infty}(t) \sqrt{t}} \int _0^t \sigma _\infty (s) dB_s + \frac{1}{\overline{\sigma _\infty}(t) \sqrt{t}}\int _0^t (\sigma (s,X_s^x) -\sigma _\infty (s))dB_s \\
&\quad + \frac{1}{\overline{\sigma _\infty}(t) \sqrt{t}} \int _0^t (b(s,X_s^x ) -b_\infty (s) ) ds.
\end{align*}
Hence, it holds that
\begin{equation}\label{eq:FR}
F_t^x = \frac{1}{\overline{\sigma _\infty}(t) \sqrt{t}} \int _0^t \sigma _\infty (s) dB_s + \frac{1}{\sqrt{t}} R_t^x
\end{equation}
where
\[
R_t^x:= 
\frac{1}{\overline{\sigma _\infty}(t)} \int _0^t (\sigma (s,X_s^x) -\sigma _\infty (s) )dB_s + \frac{1}{\overline{\sigma _\infty}(t)} \int _0^t (b(s,X_s^x ) -b_\infty (s) ) ds.
\]
Then 
\begin{equation}\label{eq:expF}
E[F_t^x]=\frac{1}{\sqrt{t}}E[R_t^x]
=\frac{1}{\overline{\sigma _\infty}(t)\sqrt{t}} E\left[ \int_0^t b(s,X_s^x ) -b_\infty (s) ds \right].
\end{equation}

Let $h_{\sigma ,t}\in H$ be the function defined by
\[
u\mapsto \frac{1}{\overline{\sigma _\infty}(t)}\int _0^u \sigma _\infty (s) {\mathbf 1}_{[0,t]}(s) ds.
\]
Explicit calculation yields
\begin{equation}\label{eq:normhs}
| h_{\sigma ,t} | _H^2 = \frac{1}{\overline{\sigma _\infty}(t)^2} \int _0^\infty \sigma _\infty (s) ^2 {\mathbf 1}_{[0,t]}(s) ds = t.
\end{equation}
Note that 
\[
 \int _0^t \sigma _\infty (s) dB_s = \int _0^\infty \sigma _\infty (s) {\mathbf 1}_{[0,t]}(s) dB_s = \overline{\sigma _\infty}(t) \left\langle h_{\sigma ,t} , B \right\rangle _H
\]
where the right-hand side is defined as the Wiener integral of order $1$ (see \cite[Page 8]{Shi04}). 
Since this equality yields 
\begin{equation}\label{eq:DB}
\left\langle D \int _0^t \sigma _\infty (s) dB_s, h \right\rangle _H = \overline{\sigma _\infty}(t) \left\langle h_{\sigma ,t} , h \right\rangle _H, \quad h\in H, 
\end{equation}
we have
\begin{align}
\label{eq:DF} \langle DF_t^x , h\rangle _H &= \frac{1}{\sqrt{t}} \left\langle h_{\sigma ,t} , h \right\rangle _H + \frac{1}{\sqrt{t}} \langle DR_t^x, h\rangle _H, \quad h\in H.
\end{align}
Hence by \eqref{eq:normhs},
\begin{equation}\label{eq:DB-DF}
\begin{split}
\left\langle D \int_0^t \sigma_{\infty}(s) dB_s, DF_t^x \right\rangle_H
&= \frac{1}{\sqrt{t}} \left\langle D\int_0^t\sigma_{\infty}(s) dB_s, h_{\sigma,t} + DR_t^x \right\rangle_H \\
&=\frac{\overline{\sigma_{\infty}}(t)}{\sqrt{t}} \left\langle h_{\sigma,t}, h_{\sigma,t} + DR_t^x \right\rangle_H\\
&=\overline{\sigma_{\infty}}(t) \sqrt{t} + \frac{\overline{\sigma_{\infty}}(t)}{\sqrt{t}} \left\langle h_{\sigma,t}, DR_t^x \right\rangle_H.
\end{split}
\end{equation}
On the other hand, from \eqref{eq:FR} and the fact that $(-L)^{-1} \int _0^t \sigma _\infty (s) dB_s = \int _0^t \sigma _\infty (s) dB_s$, we have
\begin{align*}
&D (-L)^{-1} \left( F_t^x - E[F_t^x] \right) \\
&= \frac{1}{\overline{\sigma _\infty}(t)\sqrt{t}} D (-L)^{-1} \int _0^t \sigma _\infty (s) dB_s + \frac{1}{\sqrt{t}} D (-L)^{-1} (R_t^x -E[R_t^x]) \\
&= \frac{1}{\overline{\sigma _\infty}(t)\sqrt{t}} D \int _0^t \sigma _\infty (s) dB_s + \frac{1}{\sqrt{t}}  D (-L)^{-1} (R_t^x -E[R_t^x]) .
\end{align*}
Combining this with \eqref{eq:DB-DF}, we obtain 
\begin{equation*}
\begin{split}
&\left\langle D(-L)^{-1}(F_t^x-E[F_t^x]), DF_t^x \right\rangle_H\\
&= \frac{1}{\overline{\sigma_{\infty}}(t)\sqrt{t}} \left\langle D\int_0^t\sigma_{\infty}(s) dB_s, DF_t^x \right\rangle_H
+\frac{1}{\sqrt{t}} \left\langle D(-L)^{-1}(R_t^x-E[R_t^x]), DF_t^x\right\rangle_H \\
&= 1 + \frac{1}{t} \left\langle h_{\sigma,t}, DR_t^x \right\rangle_H
+ \frac{1}{\sqrt{t}} \left\langle D (-L)^{-1} (R_t^x -E[R_t^x]), DF_t^x \right\rangle_H.
\end{split}
\end{equation*}
Hence, this equality implies 
\begin{align*}
&E\left[ \left|  1- \left\langle D (-L)^{-1} \left( F_t^x - E[F_t^x] \right), DF_t^x \right\rangle _H \right| \right] \\
&\leq E\left[ \left|\frac{1}{t} \left\langle h_{\sigma,t}, DR_t^x \right\rangle _H \right| \right] + \frac{1}{\sqrt{t}} E\left[ \left| \left\langle D (-L)^{-1}(R_t^x -E[R_t^x]), DF_t^x \right\rangle _H \right| \right] \\
&\leq \frac{1}{t \overline{\sigma _\infty}(t)} E \left[ \left| \left\langle \int _0^\cdot \sigma _\infty (s) {\mathbf 1}_{[0,t]}(s) ds, DR_t^x \right\rangle _H \right| \right]\\
&\quad + \frac{1}{\sqrt{t}} E\left[ \left| D (-L)^{-1} (R_t^x -E[R_t^x]) \right| _H^2 \right] ^{\frac{1}{2}} E\left[ |DF_t^x|_H^2 \right] ^{\frac{1}{2}}.
\end{align*}
This implies that for the proof of \eqref{eq:BE1-01}, 
it is sufficient to show that for $t\ge 1$, 
\begin{align}
&\frac{1}{\sqrt{t}} E \left[ \left| \left\langle \int _0^\cdot \sigma _\infty (s) {\mathbf 1}_{[0,t]}(s) ds, DR_t^x \right\rangle _H \right| \right] \leq C, \label{eq:main2-02-01} \\
&\sup _{x\in [x_0,\infty )} E\left[ \left| D (-L)^{-1} (R_t^x -E[R_t^x]) \right| _H^2 \right] \leq C, \label{eq:main2-02-02} \\
&\sup _{x\in [x_0,\infty )} E\left[ |DF_t^x|_H^2 \right] \leq C, \label{eq:main2-02-03} \\
&\sup _{x\in [x_0,\infty )} \left| E[F_t^x] \right| \leq \frac{C}{\sqrt{t}} \label{eq:main2-02-04}.
\end{align}

Since $\alpha>1/2$ and $\beta>1$, we have $\alpha\wedge \beta>1/2$. 
Hence, in view of Proposition~\ref{prop:estbDXh}(ii), it follows that
\begin{equation*}
\begin{split}
\frac{1}{\sqrt{t}} E \left[ \left| \left\langle \int _0^\cdot \sigma _\infty (s) {\mathbf 1}_{[0,t]}(s) ds, DR_t^x \right\rangle _H \right| \right]
& = \frac{1}{\overline{\sigma_{\infty}}(t)\sqrt{t}} E \left[ \left| \left\langle \int _0^\cdot \sigma _\infty (s) {\mathbf 1}_{[0,t]}(s) ds, DS_t^x \right\rangle _H \right| \right]\\
& \leq C t^{\frac{1}{2}-\alpha\wedge\beta} \leq C.
\end{split}
\end{equation*}
We thus obtain \eqref{eq:main2-02-01}.  

From \eqref{eq:DF} we have
\begin{align*}
|DF_t^x|_H^2 &= \frac{1}{t} |h_{\sigma,t}+DR_t^x|_H^2 \leq \frac{2}{t} \left( |h_{\sigma,t}|^2 + |DR_t^x|_H^2 \right) \\
&= 2 + \frac{2}{\overline{\sigma_{\infty}}(t)^2 t} |DS_t^x|_H^2
\leq 2 +  \frac{2}{\sigma_1^2 t} |DS_t^x|_H^2
\end{align*}
and so 
\[
E\left[ |DF_t^x|_H^2\right] 
\leq 2 + \frac{2}{\sigma_1^2 t} E [ |DS_t^x|_H^2 ].
\]
Hence, \eqref{eq:main2-02-03} follows from Proposition~\ref{prop:estbDX}.

Since \eqref{eq:expF} implies
\begin{align*}
\left| E[F_t^x] \right| 
& = \frac{1}{\overline{\sigma_{\infty}}(t) \sqrt{t}} \left| E\left[ \int _0^t (b(s,X_s^x) - b_\infty (s)) ds \right] \right| \\
& \leq \frac{1}{\sigma_1 \sqrt{t}} E\left[ \left| \int _0^t (b(s,X_s^x) - b_\infty (s)) ds \right| \right],
\end{align*}
from Lemma~\ref{lem:estLpIb}(ii) with $p=1$ we have \eqref{eq:main2-02-04}.

In view of the fact that the mappings of linear operators
\[
\left\{ F\in L^2(P); E[F]=0\right\} \mathop{\longrightarrow}^{(-L)^{-1}} {\cal D}^{2,2} \mathop{\longrightarrow}^D {\cal D}^{1,2}(H) \hookrightarrow L^2(\mu ;H)
\]
are bounded (see \cite[Section 2.8.2]{NP12} or \cite[Proposition 4.11]{Shi04}), 
for the proof  of \eqref{eq:main2-02-02} it is sufficient to show
\[
E\left[ \left| R_t^x -E[R_t^x] \right| ^2 \right] \leq C, \quad t\geq 1 .
\]
This estimate follows from the fact that
\[
E\left[ \left| R_t^x -E[R_t^x] \right| ^2 \right] 
\leq 2E\left[ \left| R_t ^x \right| ^2 \right] + 2 \left| E[R_t^x] \right| ^2,
\]
Lemma~\ref{lem:estLpIb}(i) with $p=2$, 
and Lemma \ref{lem:estLpIb}(ii) with $p=1$ and $p=2$.
\end{proof}

\section{Case of unbounded drift coefficients}\label{sec:BEunbdd}

In this section we consider the same issue as in Section~\ref{sec:BE} in the case that $b(t,y)$ diverges at $y=l \in {\mathbb R}$.
Let $l \in {\mathbb R}$, and let $\sigma , b\in C^{0,1}([0,\infty ) \times (l ,\infty ))$ satisfy the following.

\begin{ass}\label{ass2}
\begin{itemize}


\item There exist constants $\sigma _1, \sigma _2\in (0,\infty )$ such that
\begin{equation}\label{ass:s11}\tag{4.$\sigma$1}
\sigma _1 \leq \sigma (t,y) \leq \sigma _2, \quad (t,y)\in [0,\infty ) \times (l, \infty ) .
\end{equation}

\item There exist constants $\alpha \in (0, \infty )$ and $\sigma _3$ such that
\begin{equation}\label{ass:s21}\tag{4.$\sigma$2}
\sup _{t\in [0,\infty )} \left| \frac{\partial}{\partial y} \sigma (t,y)\right| \leq \sigma _3 (y\vee 1)^{-\alpha -1}, \quad y \in (l,\infty ).
\end{equation}

\item For any $\tilde{y}\in (l ,\infty )$ it holds that
\begin{equation}\label{ass:b01}\tag{4.$b$0}
\sup _{(t,y) \in [0,\infty )\times [\tilde{y},\infty )}\left( |b(t,y)| + \left| \frac{\partial}{\partial y}b(t,y)\right| \right) < \infty.
\end{equation}

\item There exists a constant $b _1 \in (0,\infty )$ such that
\begin{equation}\label{ass:b11}\tag{4.$b$1}
b _1 \leq b (t,y) , \quad (t,y)\in [0,\infty )\times (l, \infty ) .
\end{equation}

\item There exist constants $\beta \in (0, \infty )$ and $b_3$ such that
\begin{equation}\label{ass:b21}\tag{4.$b$2}
\sup _{t\in [0,\infty )} \left| \frac{\partial}{\partial y} b(t,y)\right| \leq b_3 y^{-\beta -1}, \quad y\geq (l \vee 0) +1.
\end{equation}

\item There exist constants $\gamma _1, \gamma _2 \in (1,\infty )$ and $c_1,c_2 \in (0,\infty )$ such that
\begin{align}
\label{ass:b31}\tag{4.$b$3}& \inf _{t\in [0,\infty )} b(t,y) \geq c_1 (y-l)^{-\gamma _1 }, \quad y \in (l,l+1), \\
\label{ass:b31'}\tag{4.$b$3'}& \sup _{t\in [0,\infty )}\left| \frac{\partial}{\partial y} b(t,y)\right| \leq c_2 (y-l)^{-\gamma _2}, \quad y \in (l,l+1).
\end{align}

\end{itemize}
\end{ass}

Consider the solution $X_t^x$ of the SDE:
\begin{equation}\label{eq:SDE11}\left\{\begin{array}{l}
dX_t^x = \sigma (t,X_t^x) dB_t + b(t,X_t^x)dt, \\
X_0^x =x \in (l, \infty ).
\end{array}\right.\end{equation}
In this section, we fix $x\in (l, \infty )$.
Since
\[
\sup _{y_1, y_2 \in [z, \infty ); y_1\neq y_2} \sup _{t\in [0,\infty )} \left( \frac{|\sigma (t,y_2) - \sigma (t,y_1)|}{|y_2 -y_1|} + \frac{|b(t,y_2) - b(t,y_1)|}{|y_2 -y_1|} \right) < \infty
\]
holds for any $z \in (l,\infty )$, the solution $X_t^x$ exists pathwise-uniquely up to the explosion time 
\[
\zeta^x := \lim _{n\rightarrow \infty} \inf \left\{ t\geq 0: X_t^x \leq l + \frac{1}{n} \right\} >0.
\]
In view of \eqref{ass:b31}, by comparison with Bessel processes we have $\zeta ^x =\infty$, i.e. $X_t^x$ does not hit $l$ almost surely.

Similarly to the argument in Section~\ref{sec:BE}, 
there exist $\sigma _\infty \in C([0,\infty ); [\sigma _1, \sigma _2])$, $b_2\in [b_1,\infty)$, 
$b_\infty \in C([0,\infty ); [b_1, b_2])$ and $b_3\in [0,\infty)$ such that
\begin{align}
\label{eq:ests1} & \sup _{t\in [0,\infty )} |\sigma (t,y)-\sigma _\infty (t)| \leq \frac{\sigma _3 }{\alpha} y^{-\alpha}, \quad y\geq (l \vee 0) +1,\\
\label{eq:estb1} & \sup _{t\in [0,\infty )} |b(t,y)-b_\infty (t)| \leq \frac{b_3}{\beta} y^{-\beta}, \quad y\geq (l \vee 0) +1.
\end{align}
Similarly to Section~\ref{sec:BE}, let
\[
\overline{\sigma _\infty}(t) := \sqrt{\frac{1}{t} \int _0^t \sigma _\infty (s)^2 ds} \quad \mbox{and} \quad \overline{b _\infty}(t) := \frac{1}{t} \int _0^t b_\infty (s) ds.
\]
From the comparison principle (see \cite[Theorem 1.3]{Ya73}) and \eqref{ass:b11}, we also have
\begin{equation}\label{eq:comparison2}
Y_t^x \leq X_t^x, \quad t\in [0,\infty ) \quad \mbox{almost surely}
\end{equation}
where $Y_t^x$ is the pathwise-unique solution of \eqref{eq:SDEY}.

We have a version of the law of large numbers as follows.

\begin{thm}\label{thm:main12}
Under Assumption~{\rm \ref{ass2}}, it holds that
\[
\lim _{t\rightarrow \infty} \left( \frac{X_t^x}{t} - \overline{b _\infty}(t) \right) = 0
\]
almost surely.
\end{thm}

\begin{proof}
Similarly to the proof of Theorem~\ref{thm:main1} we obtain the assertion.
\end{proof}

Let us present  the Berry-Esseen bound for the solution $X^x$ to the SDE \eqref{eq:SDE11} as follows.

\begin{thm}\label{thm:BE2}
Let Assumption~{\rm \ref{ass2}} with $\alpha >\frac{1}{2}$ and $\beta >1$ hold.
Let $x\in (l,\infty )$, and let $X_t^x$ be the solution to \eqref{eq:SDE11}.
Furthermore, assume that there exists $q\in (1,\infty )$ satisfying
\begin{equation}\label{ass:sb1}
\sup _{(t,y)\in [0,\infty )\times {\mathbb R}} \left( \frac{q(q+1)}{q-1} \left( \frac{\partial \sigma}{\partial y} (t,y) \right) ^2 +2q \frac{\partial b}{\partial y}(t,y) \right) < \frac{\sigma _1^2 b_1^2}{2}.
\end{equation}
Then there exists a positive constant $C$ such that 
\[
d_{\rm  TV} \left( {\rm Law} \left(\frac{X_t^x -x - t \overline{b _\infty}(t)}{\overline{\sigma _\infty}(t) \sqrt{t}} \right) , {\rm Law}(Z)\right) \leq \frac{C}{\sqrt{t}}, \quad t\geq 1,
\]
where $Z$ is a random variable with the standard normal distribution.
In particular, it holds that
\[
\sup _{y\in {\mathbb R}} \left| P\left( \frac{X_t^x - x - t \overline{b _\infty}(t)}{\overline{\sigma _\infty}(t) \sqrt{t}} \geq y \right) - P(Z\geq y) \right| \leq \frac{C}{\sqrt{t}}, \quad t\geq 1 .
\]
\end{thm}


Similarly to Section \ref{sec:BE}, we denote the derivatives of $\sigma (s,y)$ and $b(s,y)$ with respect to the spatial variable $y$ by $\partial _y \sigma (s,y)$ and $\partial _y b(s,y)$, respectively.
Furthermore, throughout this section, $C$ is a constant depending on $\sigma _1$, $\sigma _2$, $\sigma _3$ $b_1$, $b_2$, $b_3$, $\alpha$, $\beta$, $q$ and $x$, 
which can be different from line to line.
When the constant $C$ depends on another parameter, for example $p$, we denote it by a subscript, like $C_p$.

\subsection{Estimates of some functionals of $X_t^x$}

Similarly to Lemma \ref{lem:expest} we have the following.

\begin{lem}\label{lem:expest2}
Let Assumption~{\rm \ref{ass2}} hold.
Then, for any $\tilde{c}_1, \tilde{c}_2\geq 0$ satisfying
\[
M := \sup _{(t,y)\in [0,\infty )\times {\mathbb R}} \left( \tilde{c}_1 [\partial _y \sigma (t,y)]^2 + \tilde{c}_2 \partial _y b(t,y) \right) < \frac{\sigma _1^2 b_1^2}{2},
\]
it holds that
\[
\sup _{x\in [x_0,\infty )} \sup _{s,t\in [0,\infty ); s<t} E\left[ \exp \left( \int _s^t \left( \tilde{c}_1 [\partial _y \sigma (v, X_v^x) ]^2 + \tilde{c}_2 \partial _y b(v, X_v^x) \right) dv \right) \right] \leq C.
\]
\end{lem}

\begin{proof}
Let $p>1$ satisfy $ pM < \frac{\sigma_1^2b_1^2}{2}$, and let $q = \frac{p}{p-1} $. 
Take $\delta>0$ so small that $q \delta < \frac{\sigma_1^2 b_1^2}{2}$. 
Let $\gamma=(2\alpha) \wedge \beta$. 
Then it follows by \eqref{ass:s21} and \eqref{ass:b21} that, 
for any $\varepsilon\in (0,\gamma)$ 
there exists $L \in [ (l\vee 0)+1, \infty )$, which depends on $\sigma_1$, $\sigma_3$, $b_1$ and $b_3$, such that 
\[
\sup _{t \in [0,\infty )} \left( \tilde{c}_1 [\partial _y \sigma (t,y)]^2 + \tilde{c}_2 \partial _y b(t,y) \right) 
\leq \delta y ^{-(1+\gamma -\varepsilon)}, \quad y \in [L,\infty ).
\]
Hence, similarly to the proof of Lemma~\ref{lem:expest}, we obtain the assertion. 
\end{proof}

Now we prepare some estimates on the behaviour of $X^x$ around $l$. 

\begin{prop}\label{prop:intX}
Under Assumption~{\rm \ref{ass2}}, for any $\gamma \in (0,\infty )$ it holds that
\[
E\left[ \sup _{t\in [0,\infty )} (X_t^x -l)^{-\gamma }\right] < \infty .
\]
\end{prop}

\begin{proof}
It is sufficient to show that
\[
E\left[ \sup _{t\in [0,\infty )} (X_t^x -l)^{-\gamma } h(X_t^x)\right] < \infty
\]
where $h \in C^\infty ({\mathbb R})$ satisfies  ${\bf 1}_{(l,l+1]} \leq h \leq {\bf 1}_{(l,l+2]}$.
By It\^o's formula we have
\begin{align*}
&(X_t^x-l)^{-\gamma}h(X_t^x) - (x-l)^{-\gamma}h(x)\\
&= \int _0^t \left( -\gamma (X_s^x-l)^{-\gamma -1}h(X_s^x) + (X_s^x-l)^{-\gamma}h'(X_s^x) \right) \sigma (s,X_s^x) dB_s  \\
&\quad + \int _0^t \left( \frac{1}{2} \gamma (\gamma +1) (X_s^x-l)^{-\gamma -2} \sigma (s,X_s^x)^2 - \gamma (X_s^x-l)^{-\gamma -1} b(s,X_s^x) \right) h(X_s^x) ds \\
&\quad + \int _0^t R(s,X_s^x) ds
\end{align*}
where
\[
R(s,y):= \frac{1}{2}\left( (y-l)^{-\gamma} h''(y) - 2\gamma (y-l)^{-\gamma -1}h'(y)  \right) \sigma (s,y)^2 + (y-l)^{-\gamma} h'(y) b(s,y).
\]
In view of \eqref{ass:b31} there exists $\varepsilon >0$ such that
\begin{equation}\label{eq:propintX01}
\frac{1}{2} \gamma (\gamma +1) (y-l)^{-\gamma -2} \sigma (s,y)^2 - \gamma (y-l)^{-\gamma -1} b(s,y) \leq 0 , \quad y\in (l,l+\varepsilon ),
\end{equation}
and the Burkholder-Davis-Gundy inequality implies
\begin{align*}
&E\left[ \sup _{t\in [0,\infty)} \int _0^t \left( -\gamma (X_s^x-l)^{-\gamma -1}h(X_s^x) + (X_s^x-l)^{-\gamma}h'(X_s^x) \right) \sigma (s,X_s^x) dB_s\right] \\
&\leq C E\left[ \left( \int _0^\infty \left( -\gamma (X_s^x-l)^{-\gamma -1} h(X_s^x) + (X_s^x-l)^{-\gamma} h' (X_s^x) \right) ^2 \sigma (s,X_s^x)^2 ds \right) ^{\frac{1}{2}}\right] \\
&\leq C\left( 1+ \gamma ^2 \sigma _2^2 E\left[ \int _0^\infty \left( (X_s^x-l)^{-2\gamma -2} h(X_s^x)^2 + (X_s^x-l)^{-2\gamma} h' (X_s^x)^2 \right) ds \right] \right) .
\end{align*}
Since $h^2\le h$ and the supports of $h'$ and $h''$ are included in the interval $[l+1,l+2]$, we have
\begin{equation}\label{eq:propintX02}\begin{array}{l}
\displaystyle E\left[ \sup _{t\in [0,\infty)}  (X_t^x-l)^{-\gamma} h(X_t^x)\right] - (x-l)^{-\gamma} h(x)\\
\displaystyle \leq C\left( 1+ \int _0^\infty E\left[ (X_s^x-l)^{-2\gamma -2} h(X_s^x) \right] ds + \int _0^\infty P(X_s^x \leq l+2) ds \right) .
\end{array}\end{equation}
Let $p \in (1,\infty )$.
Note that H\"older's inequality and ${\bf 1}_{(l,l+1]} \leq h \leq {\bf 1}_{(l,l+2]}$ imply
\begin{equation}\label{eq:propintX10}\begin{array}{l}
\displaystyle \int _0^\infty E\left[ (X_s^x-l)^{-2\gamma -2} h(X_s^x) \right] ds \\
\displaystyle \leq \left( \sup _{s\in [0,\infty)} E\left[ (X_s^x-l)^{-2p(\gamma +1)} h(X_s^x) \right] \right) ^{\frac{1}{p}} \left( \int _0^\infty P(X_s^x \leq l+2)^{\frac{p-1}{p}} ds\right) .
\end{array}\end{equation}
We apply a similar argument to above.
By It\^o's formula we have
\begin{align*}
&E\left[ (X_t^x-l)^{-2p(\gamma +1)}h(X_t^x) \right] - (x-l)^{-2p(\gamma +1)}h(x) \\
&=  \int _0^t E\left[ \left( p (\gamma +1)(2p(\gamma +1)+1) (X_s^x-l)^{-2p(\gamma +1)-2} \sigma (s,X_s^x)^2 \right. \right. \\
&\quad \hspace{4cm} \left. \left. \phantom{\frac{1}{2}} - 2p (\gamma +1) (X_s^x-l)^{-2p(\gamma +1)-1} b(s,X_s^x) \right) h(X_s^x) \right] ds \\
&\quad + \int _0^t E\left[ \tilde{R}(s,X_s^x) \right] ds
\end{align*}
where
\begin{align*}
\tilde{R}(s,y) &:= \frac{1}{2}\left[ (y-l)^{-2p(\gamma +1)} h''(y)- 4p(\gamma +1)(y-l)^{-2p(\gamma +1)-1}h'(y) \right] \sigma (s,y)^2 \\
&\qquad + (y-l)^{-2p(\gamma +1)} h'(y) b(s,y).
\end{align*}
Similarly to \eqref{eq:propintX01} there exists $\tilde{\varepsilon} >0$ such that
\begin{equation}\label{eq:propintX03}\begin{array}{l}
\displaystyle p (\gamma +1)(2p(\gamma +1)+1) (y-l)^{-2p(\gamma +1)-2} \sigma (s,y)^2 \\[1mm]
\displaystyle \hspace{2.5cm} - 2p(\gamma +1) (y-l)^{-2p(\gamma +1)-1} b(s,y) \leq 0 , \quad y\in (l,l+\tilde{\varepsilon} ).
\end{array}\end{equation}
Hence, in view of the supports of $h$, $h'$ and $h''$ it follows that
\[
E\left[ (X_t^x-l)^{-2p(\gamma +1)} h(X_t^x)\right] - (x-l)^{-2p(\gamma +1)} h(x) \leq \tilde{C} \left( 1+  \int _0^t P(X_s^x \leq l+2) ds \right) .
\]
From this inequality, \eqref{eq:propintX02} and \eqref{eq:propintX10} we obtain
\begin{align*}
&E\left[ \sup _{t\in [0,\infty)}  (X_t^x-l)^{-\gamma} h(X_t^x)\right] \\
&\leq (x-l)^{-\gamma} h(x) + C \left( 1+ \int _0^\infty P(X_s^x \leq l+2) ds \right) \\
&\quad +C \left( (x-l)^{-2p(\gamma +1)} h(x) + \tilde{C} \left( 1+  \int _0^\infty P(X_s^x \leq l+2) ds \right) \right) ^{\frac{1}{p}} \\
&\quad \hspace{6cm}\times \left( \int _0^\infty P(X_s^x \leq l+2)^{\frac{p-1}{p}} ds\right) .
\end{align*}
Since \eqref{eq:comparison2} and Proposition~\ref{prop:asymY}(i) imply
\[
\int _0^\infty P(X_s^x \leq l+2)^{\frac{p-1}{p}} ds < \infty ,
\]
we obtain the conclusion.
\end{proof}

By applying Proposition \ref{prop:intX} we prove a similar lemma to Lemma \ref{lem:estLpIb}.

\begin{lem}\label{lem:estLpIb2}
\begin{enumerate}
\item[\rm (i)]
Suppose that \eqref{ass:s11} is satisfied with some $\alpha >\frac{1}{2}$. 
Then for any $p\in [1,\infty )$, 
\[
\sup _{x\in [x_0, \infty )} \sup _{t\in [0,\infty )} E\left[ \left| \int _0^t \left( \sigma (s,X_s^x) - \sigma _\infty (s) \right) dB_s \right| ^p \right] \leq C_p .
\]

\item[\rm (ii)]
Suppose that \eqref{ass:b11} is satisfied with some $\beta >1$. 
Then for any $p\in [1,\infty )$, 
\[
\sup _{x\in [x_0, \infty )} \sup _{t\in [0,\infty )} E\left[ \left| \int _0^t \left( b(s,X_s^x) - b_\infty (s) \right) ds \right| ^p \right] \leq C_p .
\]
\end{enumerate}
\end{lem}

\begin{proof}
The proof of (i) is the same as that of Lemma~\ref{lem:estLpIb}(i).
So, we omit it.

We prove (ii).
From \eqref{ass:b21}, \eqref{ass:b31'} and \eqref{eq:comparison2} we have  
\begin{align*}
&E\left[ \left| \int _0^t \left( b(s,X_s^x) - b_\infty (s) \right) ds \right| ^p \right] ^{\frac{1}{p}}\\
&\leq E\left[ \left| \int _0^t \left( \int _{X_s^x}^\infty \partial _y b(s,y) dy \right) ds \right| ^p \right] ^{\frac{1}{p}}\\
&\leq E\left[ \left( \int _0^t \left( \int _{X_s^x}^{l+1} c_2(y-l)^{-\gamma _2} dy \right) {\bf 1}_{(l,l+1)}(X_s^x) ds \right) ^p \right] ^{\frac{1}{p}} \\
&\quad + 2 \left( \sup _{(s,y)\in [0,\infty )\times [l+1,\infty )} |b(s,y)| \right) E\left[ \left( \int _0^t {\bf 1}_{[l+1, (l \vee 0) +1)}(X_s^x) ds \right) ^p \right] ^{\frac{1}{p}} \\
&\quad + E\left[ \left( \int _0^t \left( \int _{X_s^x}^{\infty} b_3 y^{-\beta -1} dy \right) {\bf 1}_{[(l \vee 0) +1,\infty)}(X_s^x) ds \right) ^p \right] ^{\frac{1}{p}} \\
&\leq \frac{c_2}{\gamma _2 -1} E\left[ \left( \int _0^t (X_s^x-l)^{-\gamma _2+1} {\bf 1}_{(-\infty ,l+1)}(Y_s^x) ds \right) ^p \right] ^{\frac{1}{p}} \\
&\quad + 2 \left( \sup _{(s,y)\in [0,\infty )\times [l+1,\infty )} |b(s,y)| \right) E\left[ \left( \int _0^t {\bf 1}_{(-\infty ,(l \vee 0) +1)}(Y_s^x) ds \right) ^p \right] ^{\frac{1}{p}} \\
&\quad + \frac{b_3}{\beta} E\left[ \left( \int _0^t (Y_s^x)^{-\beta} {\bf 1}_{[1,\infty)}(Y_s^x) ds + \int _0^t {\bf 1}_{(-\infty , 1)}(Y_s^x) ds \right) ^p \right] ^{\frac{1}{p}} .
\end{align*}
Since H\"older's inequality implies
\begin{equation}\label{eq:estLpIb2-01}
\begin{split}
&E\left[ \left( \int _0^t (X_s^x-l)^{-\gamma _2+1} {\bf 1}_{(-\infty ,l+1)}(Y_s^x) ds \right) ^p \right] \\
&\leq E\left[ \left( \sup _{s\in [0,t]} (X_s^x-l)^{-\gamma _2+1} \right) ^{pq} \right] ^{\frac{1}{q}} E\left[ \left( \int _0^t {\bf 1}_{(-\infty ,l+1)}(Y_s^x) ds \right) ^{\frac{pq}{q-1}} \right] ^{\frac{q-1}{q}}
\end{split}
\end{equation}
for $q\in (1,\infty )$, (ii) follows from Proposition~\ref{prop:intX} and Corollary~\ref{cor:intY}.
\end{proof}

\subsection{Estimates of Malliavin derivatives}

We define the Wiener space $(W,H,\mu )$ as in Section~\ref{subsec:MCbdd}.
Thanks to Lemma~\ref{lem:expest2}, we can prove the following propositions in a similar way to Propositions~\ref{prop:estbDX} and~\ref{prop:estbDXh}.
So, we omit the proofs.

\begin{prop}\label{prop:estbDX2}
Let Assumption~{\rm \ref{ass2}} and \eqref{ass:sb1} hold, and let
\[
S_t^x:= \int _0^t (\sigma (s,X_s^x) -\sigma _\infty (s))dB_s + \int _0^t (b(s,X_s^x ) -b_\infty (s)) ds.
\]
\begin{enumerate}
\item[{\rm (i)}] It holds that for any $t\in [0,1]$, 
\[
E \left[ \left| D S_t^x \right| _H^2 \right] \leq Ct.
\]
\item[{\rm (ii)}] It holds that for $t\ge 1$, 
\[
E \left[ \left| D S_t^x \right| _H^2 \right]
\le 
C\times 
\begin{dcases}
t^{1-2(\alpha \wedge \beta)} & \left( 0< \alpha \wedge \beta < \frac{1}{2}\right), \\
1+\log t & \left( \alpha \wedge \beta = \frac{1}{2} \right), \\
1 & \left( \alpha \wedge \beta > \frac{1}{2}\right) .
\end{dcases}
\]
\end{enumerate}
\end{prop}

\begin{prop}\label{prop:estbDXh2}
Let Assumption~{\rm \ref{ass2}} and \eqref{ass:sb1} hold. 
Define $S_t^x$ as in Proposition~{\rm \ref{prop:estbDX2}}.
\begin{enumerate}
\item[{\rm (i)}] 
It holds that
\[
E \left[ \left| \left\langle \int _0^\cdot \sigma _\infty (s) {\mathbf 1}_{[0,t]}(s) ds, DS_t^x \right\rangle _H \right| \right] \leq C t, \quad t\in [0,1].
\]
\item[{\rm (ii)}]
It holds that for $t\geq 1$, 
\begin{align*}
&E \left[ \left| \left\langle \int _0^\cdot \sigma _\infty (s) {\mathbf 1}_{[0,t]}(s) ds, DS_t^x \right\rangle _H \right| \right] \leq C \times 
\begin{dcases}
t^{1-(\alpha \wedge \beta)} & \left( 0< \alpha \wedge \beta < 1\right), \\
1+\log t & \left( \alpha \wedge \beta = 1 \right), \\
1 & \left( \alpha \wedge \beta > 1\right) .
\end{dcases}
\end{align*}
\end{enumerate}
\end{prop}

\begin{rem}
Since $X_t^x$ satisfies \eqref{eq:SDE11}, if $DX_t^x$ is well-defined as an $H$-valued random variable, 
$DS_t^x$ is also well-defined and satisfies 
\[
\langle DS_t^x , h\rangle _H = \langle DX_t^x, h\rangle _H - \int _0^t \sigma _\infty (s) dh(s) , \quad h\in H.
\]
For the existence of $DX_t^x$ see the proof of Theorem~{\rm \ref{thm:BE2}} below.
\end{rem}

\subsection{Proof of Theorem~\ref{thm:BE2}}

\begin{proof}[Proof of Theorem~{\rm \ref{thm:BE2}}]
Let 
\[
F_t^x := \frac{X_t^x - x - t \overline{b _\infty}(t)}{\overline{\sigma _\infty}(t) \sqrt{t}} , \quad t>0.
\]
In view of \eqref{eq:MS2}, it is sufficient to show that $F_t^x \in {\cal D}^{1,2}$ for $t\geq 1$ and that for some constant $C>0$,
\begin{equation}\label{eq:BE2-01}
E\left[ \left| 1- \left\langle D (-L)^{-1} \left( F_t^x - E[F_t^x] \right), DF_t^x \right\rangle _H \right| \right] + |E[F_t^x]| \leq \frac{C}{\sqrt{t}}, \quad t\geq 1. 
\end{equation}
Similarly to the proof of Theorem~\ref{thm:BE1}, we have
\[
F_t^x = \frac{1}{\overline{\sigma _\infty}(t)\sqrt{t}} \int _0^t \sigma _\infty (s) dB_s + \frac{1}{\sqrt{t}} R_t^x
\]
where
\[
R_t^x:= \frac{1}{\overline{\sigma _\infty}(t)}\int _0^t (\sigma (s,X_s^x) -\sigma _\infty (s) )dB_s + \frac{1}{\overline{\sigma _\infty}(t)} \int _0^t (b(s,X_s^x ) -b_\infty (s) ) ds.
\]
Since $X_t^x$ does not hit $l$, from \eqref{ass:s11}, \eqref{ass:s21} and \eqref{ass:b01} we see that $X_t^x \in {\cal D}_{\rm loc}^{1,2}$ and in particular we are able to apply the Malliavin calculus to $X_t^x$ (see \cite[page 49]{Nu06}).
This implies that $F_t^x \in {\cal D}_{\rm loc}^{1,2}$.
From Proposition~\ref{prop:estbDX}(i), $F_t^x \in {\cal D}^{1,2}$ also follows.

The rest of the proof is similar to the proof of Theorem~\ref{thm:BE1}, because we can apply Propositions~\ref{prop:estbDX2} and~\ref{prop:estbDXh2} and Lemma~\ref{lem:estLpIb2}, instead of Propositions~\ref{prop:estbDX} and~\ref{prop:estbDXh} and Lemma~\ref{lem:estLpIb}, respectively.
So, we omit the proof.
\end{proof}

\subsection{Relaxation of Assumption \ref{ass2} and sharpness of Theorem \ref{thm:BE2}}
\label{sec:relax}
In this subsection, we examine the sharpness of Theorem~\ref{thm:BE2} 
by applying it to the Brownian motion on a hyperbolic space. 
In order to do so, we concern the relaxation of the condition \eqref{ass:b31} in Assumption~ \ref{ass2}. 
More precisely, we replace the condition \eqref{ass:b31} by  
\begin{equation}\label{ass:b31-1}\tag{4.$b$3-1}
\sup_{t\in [0,\infty)} b(t,y) \geq c_1(y-l)^{-1}, \quad y\in (l,l+1),
\end{equation}
that is, we assume that \eqref{ass:b31} is satisfied with $\gamma_1=1$. 
We begin by explaining how this replacement affects the argument for the proof of Theorem \ref{thm:BE2}. 

In what follows, we assume that \eqref{ass:b31-1} is fulfilled. 
We first reconsider the proof of Proposition~\ref{prop:intX}. 
Under the condition \eqref{ass:b31-1}, 
we only need to change the argument in \eqref{eq:propintX01} and \eqref{eq:propintX03}. 
In fact, \eqref{eq:propintX01} is valid if  
$0<\gamma < \frac{2c_1}{\sigma_2^2}-1$, and 
\eqref{eq:propintX03} is valid if 
$0<\gamma \leq \frac{1}{2p} \left( \frac{2c_1}{\sigma_2^2}-1 \right) -1$.
Moreover, the proof of Proposition~\ref{prop:intX} works for any $p\in (1,\infty )$.
Hence, the assertion of Proposition~\ref{prop:intX} holds 
when $\frac{2c_1}{\sigma_2^2}>3$ and 
\begin{equation}\label{eq:gamma_1}
0<\gamma < \frac{1}{2}\left( \frac{2c_1}{\sigma_2^2}-3 \right) .
\end{equation}

We next reconsider the proof of Lemma~\ref{lem:estLpIb2}(ii). 
To show that the last expression in \eqref{eq:estLpIb2-01} is bounded,  
Proposition~\ref{prop:intX} is used with $\gamma=pq(\gamma_2-1)$.
Noting that the proof of Lemma~\ref{lem:estLpIb2}(ii) works for any $q\in (1,\infty )$ and taking \eqref{eq:gamma_1} into account, we see that when $\frac{2c_1}{\sigma_2^2}>3$ and 
\[
1<\gamma_2 < 1+\frac{1}{2p}\left(\frac{2c_1}{\sigma_2^2}-3\right),
\]
then the assertion of Lemma~\ref{lem:estLpIb2}(ii) still holds.

We finally note that, for the proof of Theorem~\ref{thm:BE2}, 
we apply Lemma~\ref{lem:estLpIb2}(ii) with $p=1$ and $p=2$. 
Therefore, when $\frac{2c_1}{\sigma_2^2}>3$ and 
\begin{equation}\label{eq:gamma_2}
1<\gamma_2 < 1+\frac{1}{4}\left(\frac{2c_1}{\sigma_2^2}-3\right),
\end{equation}
then Lemma~\ref{lem:estLpIb2}(ii) is valid with $p=1$ and $p=2$, 
and so Theorem~\ref{thm:BE2} holds. 

As a consequence of the argument above, 
we have the following: assume that 
\begin{itemize}
\item[(a)] Assumption~\ref{ass2} except \eqref{ass:b31} holds, and instead, \eqref{ass:b31-1} holds,
\item[(b)] $\frac{2c_1}{\sigma_2^2}>3$, and 
the constant $\gamma_2$ in \eqref{ass:b31'} satisfies \eqref{eq:gamma_2}. 
\end{itemize}
Then, Theorem \ref{thm:BE2} holds. 

As an application of the argument above, 
we discuss the Berry-Esseen bound for the Brownian motion on a hyperbolic space. 
For $d\geq 2$, let ${\mathbb H}^d$ be the $d$-dimensional hyperbolic space with a pole $o$,
and let $d=d_{{\mathbb H}^d}$ be the associated distance.
Let $X=(\{X_t\}_{t\ge 0},\{P_x\}_{x\in {\mathbb H}^d})$ be the Brownian motion on ${\mathbb H}^d$
generated by the half of the Laplace-Beltrami operator.
Let $R_t^{(d)}=d(o,X_t)$ be the radial process and $P=P_o$.
We know that $R_t^d$ satisfies the stochastic differential equation 
\begin{equation}\label{eq:HBM}
R_t^{(d)}=B_t+\frac{d-1}{2}\int_0^t \coth R_s^{(d)} ds,
\end{equation}
where $B_t$ is the one-dimensional Brownian motion (see \cite{Shi23} and the references therein for details).

In order to apply Theorem \ref{thm:BE2} to \eqref{eq:HBM}, 
we take 
\[
\sigma(t,y)=1, \quad b(t,y)=\frac{d-1}{2}\coth y, \quad l=0.
\]  
Since
\[
b(t,y)\ge \frac{d-1}{2y} \quad (y>0)
\]
and 
\[
\partial_y b(t,y)=-\frac{d-1}{2\sinh^2 y}\ge  -\frac{d-1}{2y^2} \quad (y>0), 
\]
we have \eqref{ass:b31-1} with $c_1=\frac{d-1}{2}$. 
We can also verify Assumption~\ref{ass2} except \eqref{ass:b31} with $\sigma_1=\sigma_2=1$ and $\gamma_2=2$.
Therefore, (a) is satisfied. 
If we further assume that $d>8$, then \eqref{eq:gamma_2} is satisfied with $\gamma_2=2$ 
and so (b) is satisfied.  
We can thus apply Theorem~\ref{thm:BE2} to the radial process $R_t^{(d)}$ for $d> 8$.
On the other hand, in \cite{Shi23} it is proved that 
the assertion is true for $d\ge 2$, 
together with the sharpness of the convergence rate $t^{-\frac{1}{2}}$ for odd dimensions or $d=2$.  
This means that the order  $t^{-\frac{1}{2}}$ in Theorem \ref{thm:BE2} is sharp.

The reason why we need the stronger condition $d>8$ in order to apply the result in the present paper seems that we consider the general cases.
Indeed, because of the diffusion coefficient $\sigma$, we have to apply the Burkholder-Davis-Gundy inequality, which creates some moments different from those coming from drift terms.
Moreover, if $\sigma$ is just a constant and $b$ is non-increasing, the assumption \eqref{ass:sb1} is trivial and the argument will be simplified.
Since the purpose of the present paper is considering general cases, we do not discuss further details.

\section{Under the assumptions on the asymptotics of coefficients}\label{sec:BEasypm}

In this section we discuss the Berry-Esseen bound under the assumptions only on the asymptotics of $\sigma$ and $b$.
Let $l \in {\mathbb R}\cup \{ -\infty \}$ and $\sigma , b\in C^{0,1}([0,\infty ) \times (l ,\infty ))$. 
We now impose the assumptions on $\sigma$ and $b$ as follows.

\begin{ass}\label{ass3}
\begin{itemize}


\item There exist constants $\sigma _1, \sigma _2\in (0,\infty )$ such that
\begin{equation}\label{ass:s12}\tag{5.$\sigma$1}
\sigma _1 \leq \sigma (t,y) \leq \sigma _2, \quad (t,y)\in [0,\infty ) \times (l, \infty ) .
\end{equation}

\item There exists a constant $\alpha \in (0, \infty )$ such that
\begin{equation}\label{ass:s22}\tag{5.$\sigma$2}
\limsup_{y\rightarrow\infty }y^{\alpha+1}\sup _{t\in [0,\infty )} \left| \frac{\partial}{\partial y} \sigma (t,y)\right|<\infty.
\end{equation}


\item For any $\tilde{y}\in (l ,\infty )$ it holds that
\begin{equation}\label{ass:b02}\tag{5.$b$0}
\sup _{(t,y) \in [0,\infty )\times [\tilde{y},\infty )}\left( |b(t,y)| + \left| \frac{\partial}{\partial y}b(t,y)\right| \right) < \infty.
\end{equation}

\item There exists a constant $b _1 \in (0,\infty )$ such that
\begin{equation}\label{ass:b12}\tag{5.$b$1}
b _1 \leq b (t,y) , \quad (t,y)\in [0,\infty )\times (l, \infty ) .
\end{equation}

\item There exists a constant $\beta \in (0, \infty )$ such that
\begin{equation}\label{ass:b22}\tag{5.$b$2}
\limsup_{y\rightarrow \infty} y^{\beta +1} \sup _{t\in [0,\infty )} \left| \frac{\partial}{\partial y} b(t,y)\right| <\infty.
\end{equation}

\end{itemize}
\end{ass}

Consider the solution $X_t^x$ of the SDE:
\begin{equation}\label{eq:SDE12}\left\{\begin{array}{l}
dX_t^x = \sigma (t,X_t^x) dB_t + b(t,X_t^x)dt, \\
X_0^x =x \in (l,\infty ).
\end{array}\right.\end{equation}
Similarly to Section~\ref{sec:BEunbdd}, the solution $X_t^x$ exists pathwise-uniquely up to the explosion time and we assume the explosion time is infinity.
Also, similarly to the argument in Section~\ref{sec:BE}, by \eqref{ass:s22} and \eqref{ass:b22} 
we have $\sigma _\infty \in C([0,\infty ); [\sigma _1, \sigma _2])$, $b_2\in [b_1,\infty)$, $b_\infty \in C([0,\infty ); [b_1, b_2])$, $y_1\in (l \vee 0, \infty )$ and $\sigma _3, b_3\in {\mathbb R}$ such that
\begin{align}
\label{eq:ests2} & \sup _{t\in [0,\infty )} |\sigma (t,y)-\sigma _\infty (t)| \leq \frac{\sigma _3 }{\alpha} y^{-\alpha}, \quad y\geq y_1,\\
\label{eq:estb2} & \sup _{t\in [0,\infty )} |b(t,y)-b_\infty (t)| \leq \frac{b_3}{\beta} y^{-\beta}, \quad y\geq y_1.
\end{align}
Similarly to Section~\ref{sec:BE}, let
\[
\overline{\sigma _\infty}(t) := \sqrt{\frac{1}{t} \int _0^t \sigma _\infty (s)^2 ds} \quad \mbox{and} \quad \overline{b _\infty}(t) := \frac{1}{t} \int _0^t b_\infty (s) ds. 
\]

Note that by the assumption \eqref{ass:b12}, we also have \eqref{eq:comparison2}.
Now we prove the following theorem.

\begin{thm}\label{thm:BE3}
Let Assumption~{\rm \ref{ass3}} with $\alpha >\frac{1}{2}$ and $\beta >1$ hold.
Let $x\in (l,\infty )$, and let $X_t^x$ be the solution to \eqref{eq:SDE12}. 
Assume that $X_t^x$ does not hit $l$ almost surely.
Furthermore, assume that there exists a constant $q\in (1,\infty )$ such that 
\begin{equation}\label{ass:sb2}
\limsup _{y\rightarrow \infty} \sup _{t\in [0,\infty )} \left( \frac{q(q+1)}{q-1} \left( \frac{\partial \sigma }{\partial y} (t,y) \right) ^2 +2q \frac{\partial b}{\partial y} (t,y) \right) < \frac{\sigma _1^2 b_1^2}{2}.
\end{equation}
Then there exists a positive constant $C$ such that 
\[
d_{\rm  TV} \left( {\rm Law} \left(\frac{X_t^x -x - t \overline{b _\infty}(t)}{\overline{\sigma _\infty}(t) \sqrt{t}} \right) , {\rm Law}(Z)\right) \leq \frac{C\log t}{\sqrt{t}}, \quad t\geq 2,
\]
where $Z$ is a random variable with the standard normal distribution.
\end{thm}

\begin{proof}
To apply the strong Markov property of $X$, we prepare some  notations.
For $r\in [0,\infty )$ and $x\in (l,\infty )$, let $X_t^{x,r}$ be the pathwise-unique solution of the SDE:
\begin{equation}\label{eq:SDE12r}\left\{\begin{array}{l}
dX_t^{x,r} = \sigma (r+t,X_t^{x,r}) dB_t + b(r+t,X_t^{x,r})dt, \\
X_0^{x,r} =x \in (l,\infty ).
\end{array}\right.\end{equation}
Also we define $Y_t^{x,r}$ for $r\in [0,\infty )$ and $x\in (l,\infty )$ as the pathwise-unique solution of the SDE:
\begin{equation}\label{eq:SDEYr}\left\{\begin{array}{l}
dY_t^{x,r} = \sigma (r+t,Y_t^{x,r}) dB_t + b_1 dt, \\
Y_0^{x,r} =x \in {\mathbb R}.
\end{array}\right.\end{equation}
Clearly, it holds that $X_t^{x,0} = X_t^x$ and $Y_t^{x,0} = Y_t^x$ almost surely, where $Y_t^x$ is the solution of \eqref{eq:SDEY}.
We remark that $\sigma (r+\cdot, \cdot)$ and $b(r+\cdot, \cdot )$ satisfy Assumption~\ref{ass3} instead of $\sigma$ and $b$ respectively, with the same constants $\sigma _1$, $\sigma _2$, $\sigma _3$, $b_1$, $b_3$, $\alpha$ and $\beta$ for all $r\in [0,\infty )$.
Let
\[
\overline{\sigma _\infty ^r}(t) := \sqrt{\frac{1}{t} \int _0^t \sigma _\infty (r+s)^2 ds} \quad \mbox{and} \quad \overline{b _\infty ^r}(t) := \frac{1}{t} \int _0^t b_\infty (r+s) ds
\]
for $r\in [0,\infty )$.

From \eqref{ass:s22}, \eqref{ass:b02}, \eqref{ass:b22} and  \eqref{ass:sb2} we can choose $y_2\in [y_1, \infty )$ and $b_2\in [0,\infty )$ so that
\begin{align*}
&\left| \frac{\partial}{\partial y} \sigma (t,y)\right| \leq \frac{\sigma _3}{\alpha} y^{-\alpha -1} , \quad (t,y)\in [0,\infty )\times [y_2, \infty ), \\
&b (t,y) \leq b_2 , \quad (t,y)\in [0,\infty )\times [y_2, \infty ), \\
&\left| \frac{\partial}{\partial y} b(t,y)\right| \leq \frac{b_3}{\beta} y^{-\beta -1} , \quad (t,y)\in [0,\infty )\times [y_2, \infty ), \\
&\sup _{(t,y) \in [0,\infty )\times [y_2,\infty )} \left( \frac{q(q+1)}{q-1} \left( \frac{\partial \sigma }{\partial y} (t,y) \right) ^2 +2q \frac{\partial b}{\partial y}(t,y) \right) < \frac{\sigma _1^2 b_1^2}{2} .
\end{align*}
Choose $\widetilde{b}\in C_b^{0,1}([0,\infty ) \times {\mathbb R})$ which satisfies
\begin{align*}
\widetilde{b}(t,y) &= b(t,y), \quad (t,y) \in [0,\infty )\times [y_2 +2, \infty ),\\
\widetilde{b}(t,y) &\leq b(t,y), \quad (t,y) \in [0,\infty )\times (l, y_2 +2 ),
\end{align*}
and
\begin{align*}
& \sup _{(t,y)\in [0,\infty ) \times {\mathbb R}} \left( \frac{q(q+1)}{q-1} \left( \frac{\partial \sigma }{\partial y} (t,y) \right) ^2 +2q \frac{\partial \widetilde{b}}{\partial y} (t,y) \right) < \frac{\sigma _1^2 (b_1-\varepsilon )^2}{2}, \\
& \inf _{(t,y)\in [0,\infty ) \times  {\mathbb R}}\widetilde{b}(t,y) \geq b_1 -\varepsilon
\end{align*}
with sufficiently small $\varepsilon >0$.
In the case that $l\in {\mathbb R}$, for example, if $\varepsilon >0$ is chosen so that $0<\varepsilon< b_1 $ and 
\[
\sup _{(t,y)\in [0,\infty ) \times [y_2,\infty )} \left( \frac{q(q+1)}{q-1} \left( \frac{\partial \sigma }{\partial y} (t,y) \right) ^2 +2q \frac{\partial b}{\partial y} (t,y) \right) < \frac{\sigma _1^2 (b_1-\varepsilon )^2}{2} ,
\]
and $\widetilde{b}$ is defined by
\[
\widetilde{b}(t,y) = \begin{dcases}
b(t, y_2 +2\delta) - \int _{y\vee l}^{y_2+2\delta} \max \left\{ \frac{\partial}{\partial z}b(t,z),0 \right\} dz \\ \qquad - \int _{y\vee l}^{y_2+2\delta} \varphi (z) \min \left\{ \frac{\partial}{\partial z}b(t,z),0 \right\}  dz, &(t,y) \in [0,\infty )\times (-\infty , y_2 +2\delta) , \\
b(t,y), & (t,y) \in [0,\infty )\times [y_2+2\delta, \infty )
\end{dcases}
\]
where $\delta \in \left( 0, \frac{2\varepsilon }{b_1^2}\wedge 1 \right)$ and $\varphi \in C^\infty ({\mathbb R})$ is chosen so that 
${\bf 1}_{[y_2 +2\delta,\infty )} \leq \varphi \leq {\bf 1}_{[y_2 +\delta , \infty )}$, 
then the pair $( \widetilde{b}, \varepsilon )$ satisfies the conditions above.

Consider the solution $\widetilde{X}_t^{x,r}$ of \eqref{eq:SDE12r} with replacement of $b$ by $\widetilde{b}$.
Then, similarly to Section~\ref{sec:BE}, the comparison principle (see \cite[Theorem 1.3]{Ya73}) implies that for $r\in [0,\infty )$
\begin{equation}\label{eq:comparison10}
Y_t^{x,r} \leq \widetilde{X}_t^{x,r} \leq X_t^{x,r}, \quad t\in [0,\infty ) \quad \mbox{almost surely}.
\end{equation}
For $M \in ( (l+2)\vee x \vee 1 ,\infty )$, define a hitting time $\tau _{2M}$ by
\[
\tau _{2M}(w) := \inf \{ t\geq 0: w(t) =2M\} , \quad w \in C([0,\infty )).
\]
Let $L_1$ and $L_2$ be constants such that $1+\frac{\sigma _2^2}{b_1} < L_1< b_1 L_2$.
We choose $t>0$ sufficiently large so that $\log t > (y_2 +2) \vee x \vee 1$ from now on.
Then, we have
\begin{equation}\label{eq:unbdd01}\begin{array}{l}
\displaystyle \left| P\left(\frac{X_t^x -x - t \overline{b _\infty}(t)}{\overline{\sigma _\infty}(t) \sqrt{t}} \in A \right) - P(Z\in A)\right|\\
\displaystyle \leq \left| P\left(\frac{X_t^x -x - t \overline{b_\infty}(t)}{\overline{\sigma _\infty}(t) \sqrt{t}} \in A, \ L_2 \log t >\tau _{L_1\log t}(X^x) , 
\ \inf _{s\in [0,\infty )} X_{s+\tau _{L_1\log t}(X^x)}^x \geq \log t \right) \right. \\
\displaystyle \quad \hspace{10.5cm} \left. \phantom{\frac{1}{2}}- P(Z\in A) \right| \\
\displaystyle \quad + P\left( L_2 \log t \leq \tau _{L_1\log t}(X^x) \right) + P\left( \inf _{s\in [0,\infty )} X_{s+\tau _{L_1\log t}(X^x)}^x < \log t \right) .
\end{array}\end{equation}
Since $b(t,y)= \widetilde{b}(t,y)$ for $(t,y) \in [0,\infty )\times [y_2 +2, \infty )$, it holds that for $r\in [0,\infty )$ and $y_2 +2 \leq \log t \leq \tilde{x}$
\[
\inf _{s\in [0 ,\infty )} X_s^{\tilde{x},r} \geq \log t \Longrightarrow X_s^{\tilde{x},r} = \widetilde{X}_s^{\tilde{x},r},\ s\geq 0
\]
almost surely.
Hence, the strong Markov properties of $X$ and $\widetilde{X}$ imply
\begin{align*}
&P\left(\frac{X_t^x -x - t \overline{b_\infty}(t)}{\overline{\sigma _\infty}(t) \sqrt{t}} \in A, \ L_2 \log t >\tau _{L_1\log t}(X^x) , \ \inf _{s\in [0,\infty )} X_{s+\tau _{L_1\log t}(X^x)}^x \geq \log t \right) \\
&= \int _{[0,L_2\log t]} P\left(\frac{X^{L_1\log t, r}_{t-r} -x - t\overline{b_\infty}(t)}{\overline{\sigma _\infty}(t) \sqrt{t}} \in A, \ \inf _{s\in [0,\infty )} X^{L_1\log t ,r}_s \geq \log t \right) P( \tau _{L_1\log t}(X^x) \in dr) \\
&= \int _{[0,L_2\log t]} P\left(\frac{\widetilde{X}^{L_1\log t ,r}_{t-r} -x -t\overline{b_\infty}(t)}{\overline{\sigma _\infty}(t) \sqrt{t}} \in A, \ \inf _{s\in [0,\infty )} \widetilde{X}^{L_1\log t ,r}_s \geq \log t \right) P( \tau _{L_1\log t}(X^x) \in dr)  \\
&= \int _{[0,L_2\log t]} P \left(\frac{\widetilde{X}^{L_1\log t ,r}_{t-r} - L_1 \log t - (t-r)\overline{b_\infty ^r}(t-r)}{\overline{\sigma _\infty ^r}(t-r)\sqrt{t-r}} \in A_{t,r},\right. \\
&\quad \hspace{6cm} \left. \phantom{\frac{1}{2}} \inf _{s\in [0,\infty )} \widetilde{X}^{L_1\log t ,r}_s \geq \log t \right) P( \tau _{L_1\log t}(X^x) \in dr),
\end{align*}
where
\[
A_{t,r} := \frac{\overline{\sigma _\infty}(t)}{\overline{\sigma _\infty ^r}(t-r)}\sqrt{\frac{t}{t-r}}A+ \frac{x + t\overline{b_\infty}(t)}{\overline{\sigma _\infty ^r}(t-r)\sqrt{t-r}} - \frac{L_1 \log t + (t-r)\overline{b_\infty ^r}(t-r)}{\overline{\sigma _\infty ^r}(t-r)\sqrt{t-r}}.
\]
By this equality and \eqref{eq:unbdd01}, we have
\begin{equation}\label{eq:unbdd02}\begin{array}{l}
\displaystyle\left| P\left(\frac{X_t^x -x -t\overline{b_\infty}(t)}{\overline{\sigma _\infty}(t) \sqrt{t}} \in A \right) - P(Z\in A)\right|\\
\displaystyle \leq \int _{[0,L_2\log t]} \left| P \left( \frac{\widetilde{X}^{L_1\log t ,r}_{t-r} - L_1 \log t - (t-r)\overline{b_\infty ^r}(t-r)}{\overline{\sigma _\infty ^r}(t-r)\sqrt{t-r}} \in A_{t,r} \right) - P\left( Z\in A_{t,r}\right) \right| \\
\displaystyle \hspace{9cm} \times P( \tau _{L_1\log t}(X^x) \in dr) \\
\displaystyle \quad + \int _{[0,L_2\log t]} 
\left| P\left( Z\in A_{t,r}\right) - P(Z\in A) \right| P( \tau _{L_1\log t}(X^x) \in dr)\\
\displaystyle \quad + \int _{[0,L_2\log t]} P\left( \inf _{s\in [0,\infty )} \widetilde{X}^{L_1\log t ,r}_s < \log t \right) P( \tau _{L_1\log t}(X^x) \in dr) \\
\displaystyle \quad + 2 P\left( L_2\log t \leq \tau _{L_1\log t}(X^x) \right) \\
\displaystyle \quad + P\left( \inf _{s\in [0,\infty )} X_{s+\tau _{L_1\log t}(X^x)}^x < \log t \right) .
\end{array}\end{equation}
Applying Theorem~\ref{thm:BE1} to $\widetilde{X}_t^{L_1 \log t ,r}$, we see that there exists a constant $C>0$ independent of $t$ and $r$ satisfying
\[
d_{\rm TV} \left( {\rm Law} \left( \frac{\widetilde{X}^{L_1\log t ,r}_{t-r} - L_1 \log t - (t-r)\overline{b_\infty ^r}(t-r)}{\overline{\sigma _\infty ^r}(t-r)\sqrt{t-r}} \right) , {\rm Law}(Z)\right) \leq \frac{C}{\sqrt{t-r}}, \quad t-r \geq 1.
\]
This implies
\[
\left| P \left( \frac{\widetilde{X}^{L_1\log t ,r}_{t-r} - L_1 \log t - (t-r)\overline{b_\infty ^r}(t-r)}{\overline{\sigma _\infty ^r}(t-r)\sqrt{t-r}} \in A_{t,r} \right) - P\left( Z\in A_{t,r}\right) \right| \leq \frac{C}{\sqrt{t-r}}
\]
for $r\in \left[ 0, (L_2 \log t) \wedge (t-1) \right]$. Hence, we have
\begin{equation}\label{eq:cormain2-06}\begin{array}{l}
\displaystyle \int _{[0,L_2\log t]} \left| P \left( \frac{\widetilde{X}^{L_1\log t ,r}_{t-r} - L_1 \log t - (t-r)\overline{b_\infty ^r}(t-r)}{\overline{\sigma _\infty ^r}(t-r)\sqrt{t-r}} \in A_{t,r} \right) - P\left( Z\in A_{t,r}\right) \right| \\
\displaystyle \hspace{9cm} \times P( \tau _{L_1\log t}(X^x) \in dr) \\
\displaystyle \leq \frac{\sqrt{2}C}{\sqrt{t}}
\end{array}
\end{equation}
for sufficiently large $t$ so that $0\leq L_2 \log t \leq \frac{t}{2}$.

The strong Markov property of $X$ and \eqref{eq:comparison10} imply
\begin{align*}
&\int _{[0,L_2\log t]} P\left( \inf _{s\in [0,\infty )} \widetilde{X}^{L_1\log t ,r}_s < \log t \right) P( \tau _{L_1\log t}(X^x) \in dr) \\
&\quad \hspace{6cm} + P\left( \inf _{s\in [0,\infty )} X^x_{s+\tau _{L_1 \log t}(X^x)} < \log t \right) \\
&\leq \sup _{r\in [0,\infty )} \left( P\left( \inf _{s\in [0,\infty )} \widetilde{X}^{L_1\log t ,r }_s < \log t \right)  + P \left( \inf _{s\in [0,\infty )} X^{L_1\log t ,r}_s < \log t \right) \right) \\
&\leq 2 \sup _{r\in [0,\infty )} P\left( \inf _{s\in [0,\infty )} Y_s^{L_1 \log t ,r} < \log t \right) .
\end{align*}
Hence, in view of Proposition~\ref{prop:asymY2}(i) and $L_1- 1 > \frac{\sigma _2^2}{b_1}$, this yields
\begin{equation}\label{eq:cormain2-07}\begin{array}{rl}
\displaystyle \int _{[0,L_2\log t]} P\left( \inf _{s\in [0,\infty )} \widetilde{X}^{L_1\log t ,r}_s < \log t \right) P( \tau _{L_1\log t}(X^x) \in dr) &\\
\displaystyle + P\left( \inf _{s\in [0,\infty )} X^x_{s+\tau _{L_1 \log t}(X^x)} < \log t \right) &\displaystyle \leq \frac{C}{\sqrt{t}} , \quad t\geq 2.
\end{array}\end{equation}
Moreover, \eqref{eq:comparison10} implies that 
for sufficiently large $t$,
\begin{align*}
P\left( L_2 \log t \leq \tau _{L_1 \log t}(X^x) \right)
&= P\left( \sup _{s\in [0,L_2 \log t]}X_s^x \leq L_1 \log t \right) \\
&\leq P\left( \sup _{s\in [0,L_2 \log t]}Y_s^x \leq L_1 \log t \right) \\
&\leq P\left( Y_{L_2 \log t}^x \leq L_1 \log t \right).
\end{align*}
Hence, in view of Proposition~\ref{prop:asymY2}(ii) and $L_1 < b_1 L_2$ it holds that
\begin{equation}\label{eq:cormain2-08}
P\left( L_2 \log t \leq \tau _{L_1 \log t}(X^x) \right) \leq \frac{C}{\sqrt{t}}, \quad t\geq 2
\end{equation}
with a positive constant $C$ independent of $t$.
Note that for sufficiently large $t$ and $r\in (0,L_2\log t]$ it holds that
\begin{align*}
&\left| \frac{\overline{\sigma _\infty ^r}(t-r)\sqrt{t-r}}{\overline{\sigma _\infty}(t)\sqrt{t}} - 1 \right| \\
&\leq \frac{\sqrt{t}| \overline{\sigma _\infty ^r}(t-r) - \overline{\sigma _\infty}(t)| + \overline{\sigma _\infty ^r}(t-r) |\sqrt{t-r} -\sqrt{t}|} {\overline{\sigma _\infty}(t)\sqrt{t}} \\
&\leq \frac{1}{\sigma _1} \left| \sqrt{\frac{1}{t-r} \int _r^t \sigma _\infty (s)^2 ds} - \sqrt{\frac{1}{t} \int _0^t \sigma _\infty (s)^2 ds} \right| + \frac{\sigma _2}{\sigma _1} \cdot \frac{r}{\sqrt{t}(\sqrt{t} + \sqrt{t-r})}\\
&\leq \frac{1}{2 \sigma _1 ^2} \left| \frac{1}{t-r} \int _r^t \sigma _\infty (s)^2 ds - \frac{1}{t} \int _0^t \sigma _\infty (s)^2 ds \right| + \frac{\sigma _2r}{\sigma _1t} \\
&\leq \frac{\sigma _2^2 r}{\sigma _1^2 (t-r)} + \frac{\sigma _2r}{\sigma _1t} \leq \frac{4\sigma _2^2 L_2}{\sigma _1^2} t^{-1}\log t.
\end{align*}
Since
\[
P\left( Z\in A_{t,r} \right) = P\left( \frac{\overline{\sigma _\infty ^r}(t-r)\sqrt{t-r}}{\overline{\sigma _\infty}(t)\sqrt{t}} Z + \eta _{t,r} \in A \right)
\]
where
\[
\eta _{t,r} := \frac{x + t\overline{b_\infty}(t)}{\overline{\sigma _\infty}(t)\sqrt{t}} - \frac{L_1 \log t + (t-r)\overline{b_\infty ^r}(t-r)}{\overline{\sigma _\infty }(t)\sqrt{t}},
\]
in view of Proposition~\ref{prop:TVG} we get 
\begin{equation}\label{eq:normal-dist-3}
\begin{split}
&\int _{[0,L_2 \log t]} 
\left| P\left( Z\in A_{t,r} \right) - P(Z\in A) \right| P( \tau _{L_1 \log t}(X^x) \in dr)\\
&\le  C \left( t^{-1}\log t + \int _{[0,L_2 \log t]} |\eta _{t,r}| P( \tau _{L_1 \log t}(X^x) \in dr) \right) .
\end{split}
\end{equation}
Since for sufficiently large $t$ and $r\in (0,L_2\log t]$
\begin{align*}
|\eta _{t,r}| &\leq \frac{1}{\overline{\sigma _\infty }(t)\sqrt{t}} \left| x + t\overline{b_\infty}(t)-  (L_1 \log t + (t-r)\overline{b_\infty ^r}(t-r)) \right| \\
&\leq \ \frac{1}{\sigma _1 \sqrt{t}}\left( |x| + L_1 \log t + \left| \int _0^t b_\infty (s)ds - \int _r^t b_\infty (s) ds\right| \right) \\
&\leq Ct^{-\frac{1}{2}}(|x|+ \log t),
\end{align*}
for a constant $C>0$ independent of $t$ and $r$, it holds that
\[
\int _{[0,L_2 \log t]} |\eta _{t,r}| P( \tau _{L_1 \log t}(X^x) \in dr) \leq Ct^{-\frac{1}{2}}(|x|+ \log t) .
\]
From this inequality and \eqref{eq:normal-dist-3}, we obtain
\begin{equation}\label{eq:cormain2-10}
\begin{split}
&\int _{[0,L_2 \log t]} 
\left| P\left( Z\in A_{t,r}\right) - P(Z\in A) \right| P( \tau _{L_1 \log t}(X^x) \in dr) \\
&\leq C t^{-\frac{1}{2}} \log t , \quad t\geq 2,
\end{split}
\end{equation}
where $C$ is a positive constant independent of $A$ and $t$.

Therefore, by \eqref{eq:unbdd02}, \eqref{eq:cormain2-06} \eqref{eq:cormain2-07}, \eqref{eq:cormain2-08} and \eqref{eq:cormain2-10} we obtain the conclusion.
\end{proof}

\section*{Acknowledgments}
This work was supported by JSPS KAKENHI Grant Numbers 19H00643, 21H00988, 22H00099, 22K18675, 23K03155, 23K20801 and 23K25773.

\appendix

\section{Appendix}\label{sec:Appendix}

\begin{prop}\label{prop:TVG}
Let $Y$ be a $d$-dimensional random vector which has the centered normal distribution with covariance matrix $V$.
Then, for any nonrandom $v\in {\mathbb R}^d$ and $a>0$, it holds that
\[
\sup _{A\in {\mathcal B}({\mathbb R}^d)} \left| P(Y\in A) - P( a Y + v \in A) \right| 
\leq 2\left| a^{d}- 1\right| 
+ C \left| V^{-\frac{1}{2}} v\right|
\]
where 
$C= \frac{1}{(2\pi )^{\frac{d}{2}}} \int _{{\mathbb R}^d} \left| y\right| _{{\mathbb R}^d}\exp \left( - \frac{1}{2} |y |_{{\mathbb R}^d} ^2\right) dy
=\frac{\sqrt{2}\Gamma((d+1)/2)}{\Gamma(d/2)}$.
\end{prop}

\begin{proof}
We have
\begin{equation}\label{eq:propTVG01}
\left| P(Y\in A) - P( a Y + v \in A) \right| \leq I_1 + I_2 + I_3
\end{equation}
where
\begin{align*}
I_1 &:= \frac{1}{(2\pi )^{\frac{d}{2}}} \left| \frac{1}{\sqrt{{\rm det}V}} - \frac{1}{\sqrt{{\rm det}(a^2V)}} \right| \int _A \exp \left(- \frac{1}{2} \langle (y-v), (a^2V)^{-1}(y-v) \rangle \right) dy, \\
I_2 &:= \frac{1}{(2\pi )^{\frac{d}{2}}} \frac{1}{\sqrt{{\rm det}V}} \left| \int _A \left[ \exp \left(- \frac{1}{2} \langle (y-v), V^{-1}(y-v) \rangle \right) \right. \right. \\
&\quad \hspace{6cm} \left. \left. - \exp \left( - \frac{1}{2} \langle (y-v), (a^2 V)^{-1}(y-v) \rangle \right) \right] dy \right|, \\
I_3 &:= \frac{1}{(2\pi )^{\frac{d}{2}}\sqrt{{\rm det}V}} \left| \int _A \left[ \exp \left(- \frac{1}{2} \langle y, V^{-1}y\rangle \right) - \exp \left( - \frac{1}{2} \langle (y-v), V^{-1}(y-v) \rangle \right) \right] dy \right| .
\end{align*}
We estimate $I_1$, $I_2$ and $I_3$. 
Since
\[
I_1 = \left| a^d - 1\right| \frac{1}{{(2\pi )^{\frac{d}{2}} \sqrt{{\rm det} (a^2V)}}} \int _A \exp \left( - \frac{1}{2} \langle (y-v), (a^2V)^{-1}(y-v) \rangle \right) dy ,
\]
we have
\begin{equation}\label{eq:propTVG02}
I_1 \leq \left| a^d - 1\right| .
\end{equation}
By the change of variables formula with $\tilde{y}=y-v$, 
\begin{align*}
I_2 &= \frac{1}{(2\pi )^{\frac{d}{2}}\sqrt{{\rm det}V}} \left| \int _A \left[ \exp \left(- \frac{1}{2} \langle \tilde{y}, V^{-1}\tilde{y} \rangle \right) - \exp \left( - \frac{a^{-2}}{2} \langle \tilde{y}, V^{-1}\tilde{y} \rangle \right) \right] d\tilde{y} \right| \\
&= \frac{1}{(2\pi )^{\frac{d}{2}}\sqrt{{\rm det}V}} \left| \int _A \left[ \int _1^{a^{-2}} \frac{\langle \tilde{y}, V^{-1}\tilde{y} \rangle}{2} \exp \left( - \frac{\theta}{2} \langle \tilde{y}, V^{-1}\tilde{y} \rangle \right) d\theta \right] d\tilde{y} \right| \\
&\leq  \frac{1}{2}\int _{1\wedge a^{-2}}^{1\vee a^{-2}} \left[ \frac{1}{(2\pi )^{\frac{d}{2}}\sqrt{{\rm det}V}} \int _{{\mathbb R}^d} \langle \tilde{y}, V^{-1}\tilde{y} \rangle \exp \left( - \frac{\theta}{2} \langle \tilde{y}, V^{-1}\tilde{y} \rangle \right) d\tilde{y} \right] d\theta \\
&=  \frac{d}{2} \int _{1\wedge a^{-2}}^{1\vee a^{-2}} \theta ^{-\frac{d+2}{2}} d\theta .
\end{align*}
we have
\begin{equation}\label{eq:propTVG03}
I_2 \leq \left| a^{d} - 1\right| .
\end{equation}
The term $I_3$ is dominated as
\begin{align*}
I_3 &= \frac{1}{(2\pi )^{\frac{d}{2}}\sqrt{{\rm det}V}} \left| \int _A \left[ \int _0^1 \langle V^{-1}(y- \theta v) , v \rangle \exp \left( - \frac{1}{2} \langle (y-\theta v), V^{-1}(y- \theta v) \rangle \right) d\theta \right] dy \right| \\
&\leq \left| V^{-\frac{1}{2}} v\right| \int _0^1 \left[ \frac{1}{(2\pi )^{\frac{d}{2}}\sqrt{{\rm det}V}} \int _{{\mathbb R}^d} \left| V^{-\frac{1}{2}} (y- \theta v) \right| \exp \left( - \frac{1}{2} \langle (y-\theta v), V^{-1}(y- \theta v) \rangle \right) dy \right] d\theta \\
&= \left| V^{-\frac{1}{2}} v\right| \frac{1}{(2\pi )^{\frac{d}{2}}} \int _{{\mathbb R}^d} \left| \tilde{y} \right| _{{\mathbb R}^d} 
\exp \left( - \frac{1}{2} |\tilde{y} |_{{\mathbb R}^d} ^2\right) d\tilde{y} \\
&= \left| V^{-\frac{1}{2}} v\right| \frac{\sqrt{2}\Gamma((d+1)/2)}{\Gamma(d/2)}.
\end{align*}
From this inequality, \eqref{eq:propTVG01}, \eqref{eq:propTVG02} and \eqref{eq:propTVG03} we have the assertion.
\end{proof}


\end{document}